\documentclass[12pt, reqno]{amsart}
\usepackage[left=1in,top=1in,right=1in,bottom=1in]{geometry}
\usepackage[utf8x]{inputenc}
\usepackage{graphicx,amsmath,amssymb,amsthm,color,tikz-cd}
\usetikzlibrary{decorations.pathreplacing}
\usepackage{amsmath}
\usepackage[hidelinks]{hyperref}
\usepackage{enumitem}

\usepackage{soul}

\newtheorem{introthm}{Theorem}

\newtheorem{introcorollary}[introthm]{Corollary}

\newtheorem*{remark*}{Remark}

\numberwithin{equation}{section}
\newtheorem{theorem}{Theorem}[section]
\newtheorem{lemma}[theorem]{Lemma}
\newtheorem{proposition}[theorem]{Proposition}

\theoremstyle{definition}
\newtheorem{definition}[theorem]{Definition}
\newtheorem{example}[theorem]{Example}

\theoremstyle{remark}
\newtheorem{remark}[theorem]{Remark}

\newcommand{\mc}{\mathcal}

\newcommand{\R}{\mathbb{R}}
\newcommand{\of}{\circ}
\newcommand{\id}{\operatorname{id}}

\newcommand{\set}[1]{\left\lbrace #1 \right\rbrace}
\newcommand{\N}{\mathbb{N}}
\newcommand{\Z}{\mathbb{Z}}
\newcommand{\abs}[1]{\left| #1 \right|}
\newcommand{\del}{\partial}
\newcommand{\ve}{\varepsilon}

\renewcommand{\vec}[1]{\ensuremath{\mathbf{#1}}}

\title[Normal forms in dimension 3]{Normal forms in a neighborhood of hyperbolic periodic orbits for flows in dimension 3}
\author{Alena Erchenko, Kurt Vinhage, Yun Yang}
\date{}

\begin{document}
\begin{abstract}
 In a neighborhood of a hyperbolic periodic orbit of a volume-preserving flow on a manifold of dimension $3$, we define and show the existence of a normal form for the generator of the flow that encodes the dynamics. If the flow is a contact flow, we show the existence of a normal form for the contact form what results in an improved normal form for its Reeb vector field. Additionally, we present a few rigidity results associated to periodic data for Anosov contact flows derived from the underlying normal form theory. Finally, we establish a new local rigidity result for contact flows on manifolds of dimension $3$ in a neighborhood of a hyperbolic periodic point by finding a new link between the roof function and the return map to a section.
\end{abstract}
\maketitle

\section{Introduction}
In the study of smooth dynamical systems, the dynamics of periodic orbits
has played an important role in the rigidity of a dynamical system associated to invariants of smooth data associated with periodic orbits (for example, \cite{de1986canonical, HuKa, de1992smooth, FH, kalinin2011livvsic, GL19,DSLVY, GLP, BEHLW}, and references therein) and in the existence of smooth coordinates in a neighborhood of the fixed point in which the dynamics is represented by a well-understood class of {\it normal} or {\it resonance forms} \cite{Ste}. In this paper, we blend together both of those features. First, we establish the existence of smooth coordinates in a neighborhood of a hyperbolic periodic orbit of a smooth flow where the vector field of the flow is in a normal form that encode information about the hyperbolicity of the orbit. Then, using those normal forms, we establish several rigidity results associated to the smooth data associated with the periodic orbits. More precisely, we show the following. (See Section~\ref{sec: applications of normal forms} for relevant definitions.)

\begin{introthm}\label{introthm: vector field}
    Let $\mc O$ be a hyperbolic periodic orbit of a $C^\infty$-flow $\Phi=(\varphi^t)_{t\in\mathbb R}$ with generating vector field $X$ on a closed 3-manifold $M$ with period $T$. If $\det(d\varphi^T(p)) = 1$ for some $p \in \mc O$, then there exist a solid torus $C$ and a $C^\infty$ embedding $h\colon C \to M$ such that $h(\R / \Z \times \set{0}) = \mc O$ and 
    $(h^{-1})_*X$ is a hyperbolic conservative resonance field. If $\Phi$ preserves a volume $\mu$, $h$ can be chosen such that $h^*\mu$ is the standard Lebesgue measure on $C$.
\end{introthm}

Moreover, if the flow is the Reeb flow associated to a contact form, then we obtain a natural normal form of the contact form itself.

\begin{introthm}
\label{thm:contact-normal}
     Let $\alpha$ be a contact form on a closed 3-manifold $M$, and assume that $\mc O$ is a hyperbolic periodic orbit for the Reeb flow induced by $\alpha$. Then there exists a solid torus $C$ and a smooth embedding $h\colon C \to M$ such that $h(\R / \Z \times \set{0}) = \mc O$ and $h^*\alpha$ is in a hyperbolic resonance form on $C$.
\end{introthm}

We note that in the topological category, the Hartman-Grobman theorem gives a complete answer for periodic points both diffeomorphisms and flows: one may always find H\"older continuous coordinates in a neighborhood of the fixed point which linearizes the dynamics (\cite{hartman1960lemma, grobman39homeomorphism}). However, the topological solution has its limitations: several global dynamical features can be surprisingly linked to invariants of smooth data associated with periodic orbits, which are not detected by H\"older coordinates.

Also, traditionally, the Birkhoff normal form has been formulated for the Poincar\'{e} return map associated with a periodic orbit, providing a simplified model of the local discrete-time dynamics. However, the return map captures only a lower-dimensional view of the flow and does not fully encode the continuous-time dynamics in a tubular neighborhood of a periodic orbit. In this paper, we study the normal form of the vector field itself around the periodic orbits. This normal form involves the whole neighborhood of the periodic orbit, not only the Poincar\'{e} return map on the section, which provides a more complete understanding of the local geometry and dynamics near periodic orbits.

Moreover, as it will be seen from the applications, the normal forms proposed in this paper indeed carry the information about the hyperbolicity in comparison to the normal forms in \cite[Lemma 2.3]{HWZ,HWZ_correction}.

\begin{introcorollary}\label{coro C}
Let $\Phi$ be a smooth Anosov contact flow on a $3$-manifold. Assume that for every periodic orbit $\gamma$, there exists a section $\Sigma_\gamma$ such that the local return time to $\Sigma_\gamma$ is tangent to a constant up to the fourth order at $\gamma$, then $\Phi$ is smoothly conjugated to the geodesic flow of a constant curvature surface, or one of its canonical time changes.
\end{introcorollary}

\begin{introcorollary}\label{coro D}
    Let $(S,g)$ be a closed smooth Riemannian surface with negative Gaussian curvature, and let $\Phi = (\varphi^t)_{t \in \R}$ be its geodesic flow on the unit tangent bundle $T^1S$. Then the metric $g$ has constant curvature if and only if for every closed geodesic $\gamma$, there exists a section $\Sigma_\gamma \subset T^1S$ such that $\varphi^T\Sigma_\gamma$ is tangent to $\Sigma_\gamma$ up to the fourth order at $\gamma$, where $T$ is the period of $\gamma$.
\end{introcorollary}

Finally, a new rigidity feature of contact flows is revealed through making the local form analysis on tubular neighborhoods, rather than the study of the Poincar\'{e} map. Imprecisely, we conclude that the roof function and the return maps each determine the contact form (and hence each other). More precisely, we obtain the following result (see Section~\ref{section: tangency definition} and Definition~\ref{def: cocycle} for relevant definitions). 

\begin{introthm}
\label{thm:linear-rigidity}
        Let $\mc O$ be a hyperbolic periodic orbit of the smooth Reeb flow $\Phi$ of a contact form $\alpha$ on a 3-manifold. Then the following are equivalent: 
    \begin{enumerate}[label=(\alph*)]
      \item $\alpha$ is a linearizable contact form.
    \item The local first return map for $\Phi$ is linearizable.
    \item The local first return map for $\Phi$ is $C^\infty$ tangent to a linear map in some coordinates.
    \item The local return time of $\Phi$ is locally cohomologous to a constant.
    \item The local return time of $\Phi$ is locally cohomologous to a function which is $C^\infty$ tangent to a constant.
    \end{enumerate}
\end{introthm}

\subsection*{Organization of the paper and outline of proofs}

In Section~\ref{sec: applications of normal forms}, we recall the theory of resonances for linear maps and give our definition of normal forms for a conservative vector field and a contact form in a neighborhood of a hyperbolic periodic point. In Section~\ref{sec: data from normal forms}, we express dynamical data (e.g., the Lyapunov exponents, the Anosov and Foulon-Hasselblatt classes, and the return time and the return map to a section) associated to the periodic point in terms of functions appearing in the normal form. In Section~\ref{sec: Preliminaries}, we recall definitions associated to Poincar\'e maps of a flow and prove the criterion of two codimension one embedded submanifolds being tangent up to some order (Lemma~\ref{tangency of sections}). 

In Section~\ref{sec: normal forms for vector fields}, we prove Theorem~\ref{introthm: vector field}. Our approach is the following. First, we use the Birkhoff normal form for the Poincar\'e map on a section to find the normal form for the return time (Section~\ref{sec: normal form for cocycles}) by solving a cohomological equation similar to methods in~\cite{KH}. In particular, we start with showing existence of a formal power series solution, then, using it, we construct a smooth solution. Afterwards, in Section~\ref{sec: coordinates}, we modify our initially chosen section so that the return time of the new section is in the resonance form. Then, we construct a coordinate system in a neighborhood of a hyperbolic periodic orbit using the coordinates on the new section. The key feature of the coordinate system on a neighborhood of a hyperbolic periodic orbit is that changing the time parameter corresponds to moving along the flow.  

In Section~\ref{sec:normal form for contact}, we prove Theorem~\ref{thm:contact-normal} by working directly with the contact form instead of its Reeb vector field. In particular, we modify the initial choice of a section so that the contact form induces the standard area form on the new section (Lemma~\ref{lemma: nice section}). As a result, we find a new connection between the return time and the Poincar\'e map (Lemma~\ref{lem:dr-formula}). Finally, we again construct a coordinate system so that changing the time parameter corresponds to moving along the flow (Section~\ref{sec:contact-build}).  In Section~\ref{sec:rigidity}, we prove some rigidity results (Corollaries \ref{coro C} and \ref{coro D}) combining the features of the normal forms and the known rigidity results recalled in Section~\ref{sec: known rigidity}.

In Section~\ref{sec: base-roof rigidity}, we establish the following general version of a local rigidity phenomenon for contact flows on manifolds of dimension $3$ by exploring the consequences of the relation between the return time and the Poincar\'e map established in Lemma~\ref{lem:dr-formula}.

\begin{introthm}\label{thm E}
    Let $\Phi$ and $\Psi$ be two $C^\infty$ Reeb flows determined by contact forms $\alpha_1$ and $\alpha_2$, respectively. Consider hyperbolic periodic orbits $\mc O_1$ of $\Phi$ and $\mc O_2$ of $\Psi$. 
    
    The following are equivalent:

    \begin{enumerate}
        \item There exist tubular neighborhoods $C_1$ and $C_2$ of $\mc O_1$ and $\mc O_2$, respectively, and a $C^\infty$ diffeomorphism $H\colon C_1 \to C_2$ such that $H^*\alpha_2 = \alpha_1$.
        \item $\mc O_1$ and $\mc O_2$ have the same prime period, and there exist two sections $\Sigma_i$, $i=1,2$, which pass through some $p_i \in \mc O_i$, respectively, with the following property. For $i = 1,2$, the Poincar\'{e} maps 
        $F_i\colon \Sigma_i \to \Sigma_i$ associated to the Reeb flows of $\alpha_i$ are conjugated by a $C^\infty$-diffeomorphism $H_\Sigma\colon \Sigma_1 \to \Sigma_2$ such that $H_\Sigma^*d\alpha_2|_{\Sigma_2} = d\alpha_1|_{\Sigma_1}$.
        \item $\mc O_1$ and $\mc O_2$ have the same Lyapunov exponents, and there exist two sections $\Sigma_i$, $i=1,2$, of $\mc O_i$ with parametrization $h_i\colon D^2_\ve \to \Sigma_i$ such that
        
        \begin{itemize}
            \item $h_i^*\alpha_i|_{\Sigma_i} = \frac{1}{2}(x \, dy - y \, dx)$, 
            \item the local return maps induced by $h_i$ are in Birkhoff-Sternberg normal form (but not necessarily the same), and
            \item if $\hat{r}_i\colon \Sigma_i \to \R_+$ is the return time function associated to the Reeb flow of $\alpha_i$, then $\hat{r}_1 \of h_1 = \hat{r}_2 \of h_2$.
        \end{itemize} 
    \end{enumerate}
\end{introthm}

We improve the above theorem for a linearizable contact form (Theorem \ref{thm:linear-rigidity}) by showing that extra conditions in Theorem~\ref{thm E} will be satisfied if the Poincar\'e map is locally linear or the return time is locally cohomologous to a constant. The distinguishable feature of linear maps is that a priori their centralizer is larger than the centralizers of nonlinear maps (see Lemma \ref{lem:std-area-form}, Remark \ref{rem:nonlin-centralizer}, and Example \ref{ex:nonlinear-centralizer}). In fact, we do not expect to be able to remove the extra conditions in Theorem~\ref{thm E}. 

\subsection*{Acknowledgments}
The authors would like to thank Martin Leguil who participated in early development of this project, and Boris Hasselblatt for enlightening conversations and interest in the project. The first author was supported by NSF grant DMS-2247230 and DMS-2552860. The second author was supported by NSF grant DMS-2405440. The third author was supported by NSF grant DMS-2000167.

\section{Resonance forms and their applications}\label{sec: applications of normal forms}

We aim to provide an analogy of Birkhoff normal form for periodic points of diffeomorphisms to periodic orbits of flows. Before doing so, we briefly recall the theory of {\it resonances} for linear maps. Throughout, we let $\bar k$ denote a multindex $\bar k=(k_1,\dots,k_n)$

\begin{definition}
\label{def:res}
    Let $A \in GL(n,\R)$ be a matrix with eigenvalues $\lambda_1,\dots,\lambda_n$ (listed with multiplicity if $A$ does not have simple spectrum). A {\it resonance} for $A$ is any equality of the form
    \[\lambda_i = \lambda^{\bar k} :=\lambda_1^{k_1}\dots \lambda_n^{k_n} \qquad\text{with}\quad i\in\{1,2,\ldots,n\}\quad\text{and}\quad \bar k=(k_1,\ldots,k_n)\in\mathbb N^n_0,\]
which we call a resonance of the form $(i;k_1,\dots,k_n)$. A formal power series $F=(F_1,\dots,F_n)$ with $F_i = \sum\limits_{\bar k\in\mathbb N^n_0} c_{\bar k}^{(i)}x_1^{k_1}\dots x_n^{k_n}$ is said to be in {\it resonance form} if for all $i$ we have $c_{\bar k}^{(i)} = 0$ whenever $(i;k_1,\dots,k_n)$ is not a resonance. That is, the power series for each $F_i$ consists only of resonance terms.
\end{definition}

Notice that when $n = 2$ and $A \in SL(2,\R)$, the resonances are all of the form $(1;k+1,k)$ and $(2;k,k+1)$. In particular, the only resonance terms for the $i$th coordinate are of the form $x_i(x_1x_2)^k$. Thus, if in resonance form, a formal power series in the variables $x$ and $y$ is of the form
\begin{equation}\label{2d resonance}
F(x,y) = (xf_1(xy),yf_2(xy)),
\end{equation}
where $f_1(z)$ and $f_2(z)$ are formal power series.

The moral of the story of normal forms is that dynamical objects can always be put in resonance form. The resonance form can be often improved to reflect additional structures of the setting (e.g., in the volume-preserving setting, we can always have $f_1(z)=\sum\limits_{k=0}^\infty\omega_kz^k$ with $\omega_0\neq0$ and $f_2(z)$ being the formal inverse of $f_1(z)$ in \eqref{2d resonance}). We make the following definitions which will be the targets for our normal forms in a neighborhood of a hyperbolic periodic point of a volume-preserving flow on manifold of dimension $3$ (see Section~\ref{sec: definitions} for relevant definitions).

\begin{definition}
\label{def:resonance}
    Let $D^2_\ve$ be an open disc of radius $\ve$ in $\R^2$ centered at $\vec{0}\in\mathbb R^2$ and $C = \R / \Z \times D^2_\ve$. We say that

    \begin{enumerate}
        \item a local area-preserving $C^\infty$ diffeomorphism $F\colon D^2_\ve \to \R^2$ is in resonance form if there exist $\lambda > 1$ and a non-zero $\omega \in C^\infty(-\ve,\ve)$ such that $\omega(0) = 1$ and $F(x,y) = (\lambda x\omega(xy),\lambda^{-1}y\omega(xy)^{-1})$,
        \item a $C^\infty$ function $r\colon D^2_\ve \to \R$ is in resonance form if there exists $\bar{r}\colon (-\ve,\ve) \to \R$ such that $r(x,y) = \bar{r}(xy)$,
        \item\label{def: resonance vector field} a conservative vector field $X$ on $C$ is in resonance form if there exist $f,g \in C^\infty(-\ve,\ve)$ such that $f(0),g(0) > 0$ and $X = f(xy)\left(\frac{\del}{\del t} + x g(xy) \frac{\del}{\del x} - y g(xy) \frac{\del}{\del y}\right)$,
        \item\label{def: resonance contact form} a contact form $\alpha$ on $C$ is in resonance form if there exists $\theta \in C^\infty(-\ve,\ve)$ such that $\theta(0) > 0$, $\theta'(0) > 0$ and $\alpha = \theta(xy) \, dt + \frac{1}{2}(x \, dy - y \, dx)$. 
    \end{enumerate}
\end{definition}

Note that each of these resonance form definitions is consistent with, and a strengthening of, Definition \ref{def:res}. Indeed, in (1) and (3), the $x$ and $y$ coefficients must match the corresponding prescribed resonance terms, but have the added restriction of being inversely proportional or opposites, respectively. In (2), this makes sense since if we think of the function as taking values in a new eigenspace of eigenvalue 1, it produces resonances of the form $\lambda^k(\lambda^{-1})^k = 1$. Finally, in the case of (4), the contact form's $dt$ term can be thought of as an eigendirection of eigenvalue 1, leading to the $xy$ term.

\subsection{Preliminaries}\label{sec: Preliminaries}

In this section we recall definitions and results related to our construction of normal forms and their applications.

\subsubsection{Flows, sections and Poincar\'{e} maps}\label{sec: sections}

Historically, normal form analysis for flows has been carried out by considering sections. Since it will be also our starting point, we recall the definitions.

Consider a smooth closed $3$-manifold $M$.
Fix a $C^\infty$-flow $\Phi=(\varphi^t)_{t\in\mathbb{R}}\colon M\rightarrow M$ with generating vector field $X$. A {\it local Poincar\'{e} section} of a flow is an embedded closed disk $\Sigma$ of codimension one such that $X(x) \not\in T_x\Sigma$ for all $x \in \Sigma$. Assume further that $\Sigma$ contains a periodic point $p$ such that if $\mc O(p)$ is the orbit of $p$, $\mc O(p) \cap \Sigma = \set{p}$. Then on a small neighborhood $U \subset \Sigma$ that contains $p$, one obtains two smooth pieces of data: a \textit{return time} $r \colon U \to \R_+$ defined so that $r(x)$ is the smallest value of $t > 0$ such that $\varphi^t(x) \in \Sigma$, and the \textit{Poincar\'e map} $F\colon U \to \Sigma$ is defined by $F(x) = \varphi^{r(x)}(x)$. 

After choosing a section, one usually applies the theory of normal forms for fixed points of diffeomorphisms, and studies the system as a suspension of the local diffeomorphism $F$ with roof function $r$. We extend this approach, by building structures on tubular neighborhoods of periodic orbits directly. In particular, we want to build coordinates on such tubular neighborhoods which put the vector fields and other objects associated to local flows in the normal forms described in Definition \ref{def:resonance}. We work in dimension $3$, where tubular neighborhoods of periodic orbits are all diffeomorphic to $C = \R / \Z \times D^2_\ve$, where $D^2_\ve$ is an open disc of radius $\ve$ centered at $\vec{0} \in \R^2$, and use the coordinate $t$ for the $\R /\Z$ component, and $x$ and $y$ for the $D^2_\ve$ components. 

\subsubsection{Tangency of embedded submanifolds}\label{section: tangency definition}

In this paper, we understand the order of tangency in  terms of Taylor polynomial (see \cite{DMP} for more details and other ways to define a tangency).

\begin{definition}\label{def: tangent function}
Let $\Sigma$ be a codimension one $C^\infty$ embedded submanifold of $M$ and $p\in M$. Let $U\subset \Sigma$ be a small neighborhood of $p$ and $f\colon U\rightarrow \mathbb R$ be a function. We say that $f$ is tangent to a constant up to the $k$-th order at $p$ if there exists a neighborhood $\tilde U$ of $\vec{0}$ in $\mathbb R^2$ and a $C^\infty$ coordinate system $h\colon (\tilde U,\vec{0})\rightarrow (U, p)$ such that
\begin{equation}\label{eq:tangent up to order 2}
f\circ h(u)-f(p)=o(|u|^k) \qquad \text{as}\qquad U\ni u\rightarrow \vec{0}.
\end{equation}
\end{definition}

We note that if a function $f$ is tangent to a constant up to the k-th order at $p$, then for any coordinate system $h'\colon (\tilde U', \vec{0})\rightarrow (U,p)$ we have expression similar to \eqref{eq:tangent up to order 2}.

\begin{definition}\label{def: tangent sections}
    Let $\Sigma_1$ and $\Sigma_2$ be two codimension one embedded submanifolds of $M$ and $p \in \Sigma_1 \cap \Sigma_2$ with smooth local charts $h_i\colon D^2_\ve \to \Sigma_i$, $i =1,2$. We say, that $\Sigma_1$ and $\Sigma_2$ are tangent up to the $k$-th order at $p$ if there exists a smooth diffeomorphism $\rho\colon D^2_\ve \to D^2_\ve$ such that $$d(h_1(u),h_2\of \rho(u)) = o(|u|^k) \qquad \text{as}\qquad U\ni u\rightarrow \vec{0},$$
    where $d$ is any metric on $M$ induced by a Riemannian metric.
\end{definition}

\begin{remark}
    This definition is independent of the Riemannian metric as any two Riemannian metrics on a compact manifold are Lipschitz equivalent. Moreover, the definition is independent of the choice of the local charts. It follows from the observation that if $\tilde h_i\colon D^2_\varepsilon\rightarrow \Sigma_i$, $i=1,2$ are another local charts, then $\tilde h_1=h_1\circ (h_1^{-1}\circ \tilde h_1)$ and $\tilde h_2\circ(\tilde h_2^{-1}\circ h_2\circ\rho\circ h_1^{-1}\circ \tilde h_1)=h_2\circ (\rho\circ h_1^{-1}\circ \tilde h_1)=h_2\circ \rho\circ (h_1^{-1}\circ \tilde h_1)$ so the diffeomorphism $\tilde\rho = \tilde h_2^{-1}\circ h_2\circ\rho\circ h_1^{-1}\circ \tilde h_1$ will work for $\tilde h_1$ and $\tilde h_2$. 
\end{remark}

\begin{lemma}\label{tangency of sections}
    Let $\Sigma_1$ and $\Sigma_2$ be two codimension one embedded submanifolds of a $3-$manifold $M$ and $p \in \Sigma_1 \cap \Sigma_2$ be such that $T_p\Sigma_1 = T_p\Sigma_2$. Fix a $C^\infty$ vector field $Z$ in a neighborhood of $p$ such that $Z(p) \not\in T_q\Sigma_i(p)$, $i=1,2$, and let $(\varphi^t)$ be a flow generated by $Z$. Let $\tau\colon \Sigma_1 \to \R$ be the function defined on a neighborhood of $p \in \Sigma_1$ such that $\varphi^{\tau(q)}(q) \in \Sigma_2$. Then $\Sigma_1$ is tangent up to $k$-th order to $\Sigma_2$ if and only if $\tau$ is tangent to a $0$ up to the $k$-th order at $p$.
\end{lemma}

\begin{remark}
Since $T_p\Sigma_1 = T_p\Sigma_2$, there exists a $C^\infty$ vector field $Z$ in a neighborhood of $p$ such that $Z(q) \not\in T_q\Sigma_i(q)$ for all $q \in \Sigma_i$, $i=1,2$.
\end{remark}

\begin{proof}[The proof of Lemma \ref{tangency of sections}]

    We first construct two local diffeomorphisms, $H_i\colon (-\ve,\ve) \times \Sigma_i \to M$ by $H_i(t,q) = \varphi^t(q)$, $i=1,2$, for some $\ve>0$. Note that $dH_i(0,p)$ is invertible as $Z(p) \not\in T_p\Sigma_i$. Then $-\tau$ is exactly the $t$-coordinate of $q\ \mapsto H_2^{-1} \of H_1(0,q)$. Furthermore, if $\tilde{U}$ is a neighborhood of $\vec{0}$ in $\R^2$ and $h_1\colon \tilde{U} \to \Sigma_1$ is a coordinate chart, then the $\Sigma_2$ coordinate of $H_2^{-1} \of H_1(0,h_1(x))$ is a coordinate chart of $\Sigma_2$. Denote this function by $\hat{h}_2$. Note that $\tau$ and $\hat{h}_2$ satisfy $\varphi^{\tau(h_1(x))}(h_1(x)) =\hat{h}_2(x)$ by definition.
 If $\tau(p)=0$ and all derivatives of $\tau$ vanish up to $k$-th order, it is clear that all derivatives of $h_1$ and $\hat{h}_2$ agree up to $k$-th order.
  
    Now, assume that there exist local charts $h_1$ and $h_2$ for $\Sigma_1$ and $\Sigma_2$ whose derivatives up to order $k$ agree. Then notice that $H_1 \of (\id\times h_1)$ and $H_2 \of (\id \times h_2)$ are $C^\infty$ tangent at $\vec 0$ as functions from $\R^3$. Thus, $\tau$, the $t$-coordinate of $H_2^{-1} \of H_1$ must be $C^\infty$ tangent to 0.
\end{proof}

\subsubsection{The Anosov property and dynamical invariants}\label{sec: definitions}
Throughout this section, let $\Phi = (\varphi_t)_{t\in\mathbb R}$ be a fixed-point free flow on a compact 3-manifold $M$. First, we recall that a hyperbolic $\Phi$-invariant set $\Lambda\subset M$ is called \textit{hyperbolic} if there exist $\lambda>0$, $C>0$ and a continuous flow-invariant splitting $T_\Lambda M=E^{uu}\oplus \mathbb RX\oplus E^{ss}$ such that for all $p\in\Lambda$ and all $t>0$
\begin{align*}
\|d\varphi^t(p)w\|&\leq Ce^{-\lambda t}\|w\|\qquad \text{for all } w\in E^{ss}(p)\\
\|d\varphi^{-t}(p)w\|&\leq Ce^{-\lambda t}\|w\|\qquad \text{for all } w\in E^{uu}(p),
\end{align*}
where $\|\cdot\|$ is the norm on $T_pM$ induced by a Riemannian metric on $M$. Subbundles $E^{uu}$ and $E^{ss}$ are called unstable and stable subbundles on $\Lambda$, respectively. Moreover, $E^{cs}=\mathbb R\oplus E^{ss}$ and $E^{cu}=\mathbb R\oplus E^{uu}$ are weak stable and unstable subbundles, respectively. We also say that $\Phi$ is an Anosov flow if $M$ itself is a hyperbolic set.

We say $\Phi$ is a Reeb flow if it is generated by the contact vector field $X_{\alpha}$ implicitly defined by the 
two conditions
$$d\alpha(X_{\alpha}, \cdot)=0, \text{ and } \alpha(X_{\alpha})=1,$$ where $\alpha$ is a contact form, i.e., a one form on $M$ with $\alpha\wedge d\alpha\neq 0$. We recall that the volume form $\alpha\wedge d\alpha$ is invariant under the flow. Moreover, if the Reeb flow for $\alpha$ is Anosov, then the kernel of $\alpha$ equals $E^{uu}\oplus E^{ss}$ and $\Phi$ is ergodic, i.e., any $\Phi$-invariant set is of full or zero measure. A classical example of an Anosov Reeb flow is the geodesic flow of manifolds with negative sectional curvatures.

If $p$ is a hyperbolic periodic orbit, consider a local Poincar\'e section $\Sigma$ endowed with local coordinates $(x,y)$, the hyperbolic periodic point $p$ being identified with the origin $\vec{0}$, and the corresponding Poincar\'e map $F$. Then $F$ is a local diffeomorphism defined in a neighborhood of $\vec{0}$, and is area-preserving. Moreover, it has a hyperbolic fixed point at $\vec{0}$. 

For sufficiently small $\Sigma$, a classical result of Sternberg \cite{Ste} provides resonance coordinates for area-preserving surface diffeormorphisms (recall Definition \ref{def:resonance}(1)):

\begin{theorem}[Birkhoff-Sternberg Coordinates]
\label{thm:sternberg}
    There exists a $C^\infty$ area-preserving change of coordinates
$
h_\Sigma \colon
\Sigma  \to \R^2$ for which $h_\Sigma\circ F\circ h_{\Sigma}^{-1}\colon D^2_\ve\rightarrow \mathbb{R}^2$ is in resonance form.
\end{theorem}

If the resonance form of $F$ is $(x,y) \mapsto (\lambda x\omega(xy),\lambda^{-1}y\omega(xy)^{-1})$, the function $\eta(z) := \log\omega(z)$ will often be more useful to work with. Consider the Taylor expansion of $\omega(z)$ at $0$, 
$$
\omega \colon z \mapsto a_0+a_1 z+  a_2 z^2+\dots,\quad \text{with }a_0 = 1.
$$
The numbers $(a_k)_{k \geq 0}=(a_k(\mathcal{O}(p)))_{k \geq 0}$ are called the \textit{Birkhoff
	invariants} or \textit{coefficients} of $F$ at $p$.

\begin{definition}[Formula (16), \cite{HuKa}] 
\label{def:anosov-cocycle}
 The number $-a_1$ is called the {\it Anosov class} of the local diffeomorphism $F$ at $p$. 
\end{definition}
\begin{remark}\label{anosov}
If $\eta(z)$ as above, the Anosov class around the periodic orbit is $-\eta'(0)$. 
\end{remark}

An analogous cocycle can be constructed by considering the return time. This cocycle was first studied in \cite{FH}, where it was called the {\it longitudinal KAM-cocycle}. To avoid confusion with KAM theory, we will refer to it as the {\it Foulon-Hasselblatt cocycle}. 

Consider {\it adapted coordinates} (constructed by Hurder and Katok) around the hyperbolic periodic point $p$ of $\Phi$ with period $T$. Let $B^3_\ve$ be a $3$-dimensional open ball of radius $\ve$ in $\mathbb R^3$. Crucially, these are collections of coordinate charts $h_x : B^3_\ve \to M$, where $x\in M$, with specific dynamical properties, most importantly, that the axes correspond to the flow, the unstable, and the stable directions (in that order). In such coordinates, one may show that the differential of $\varphi^T$ at the points of the stable leaf of a given point $p$ takes the following form:
\[ D\varphi^T(0,0,y)=\begin{bmatrix}
        1&b_1(y)&0\\
        0&\alpha^{-1}(y)&0\\
        0&b_2(y)&\alpha(y)
    \end{bmatrix},
\]
where $|\alpha(y)|<1$ and $b_i(y)=O(y)$.

\begin{definition}\label{def: Foulon-Hasselblatt class}
    The {\it longitudinal KAM obstruction} or {\it Foulon-Hasselblatt class} at the periodic orbit $p$ is $b_1'(0)$.
\end{definition}

\begin{remark}
    In fact, the Anosov class and Foulon-Hasselblatt class are values of cocycles  at the periodic point at time $T$, which are called the Anosov and Foulon-Hasselblatt cocycles, respectively. The cocycles can be constructed at arbitrary points and times of the flow, and coincide with the Anosov and Foulon-Hasselblatt classes at periodic orbits. The cocycles are well-defined, since different families of adapted coordinates yield cohomologous cocycles. See \cite{HuKa} and \cite{FH} for more details.
\end{remark}

\subsubsection{Rigidity from dynamical invariants}\label{sec: known rigidity}
The following result of Hurder-Katok says that the  Anosov class corresponds to certain  obstructions to the smoothness of the weak stable/weak unstable distributions. In view of Remark \ref{anosov}, the periodic obstructions are $-\eta'(0)$.
\begin{theorem}[Theorem 3.4, Corollary 3.5, Proposition 5.5 in \cite{HuKa}]\label{theorem HK}
	Let $\Phi$ be a $C^\infty$ transitive Anosov flow on a 3-manifold. The following properties are equivalent:
	\begin{itemize}
		\item the Anosov cocycle is a coboundary;
		\item for any periodic orbit $\mathcal{O}$, the Anosov class at  $\mathcal{O}$ vanishes;
		\item the weak stable/weak unstable distributions $E_\Phi^{cs}$/$E_\Phi^{cu}$ are $C^\infty$.
	\end{itemize} 
\end{theorem}

\begin{theorem}[\cite{FH}]\label{theorem FH}
    Let $\Phi$ be a $C^\infty$ transitive Anosov flow on a 3-manifold. The following are equivalent:
    \begin{itemize}
        \item the Foulon-Hasselblatt cocycle is a coboundary;
        \item the Foulon-Hasselblatt class vanishes at every periodic orbit of $\Phi$;
        \item the distribution $E^{ss} \oplus E^{uu}$ is $C^\infty$;
        \item $\Phi$ is either the constant-time suspension of an Anosov diffeomorphism or a contact flow.
    \end{itemize}
\end{theorem}

The smoothness obtained implies strong rigidity by the works of Ghys. In particular, we know that high regularity of the dynamical distributions implies that the flow is orbit equivalent to an algebraic model. Let $\Gamma \subset SL(2,\R)$ be a cocompact lattice, and $\tau\colon \Gamma \to \R$ be a homomorphism close to 0 (i.e., a small element of $H^1(\mathbb{H}^2/\Gamma)$). Define a new right action of $\Gamma$ on $SL(2,\R)$ by
\[  g *_\tau \gamma = \begin{pmatrix}
    e^{\tau(\gamma)/2} & 0 \\ 0 & e^{-\tau(\gamma)/2}
\end{pmatrix}g \gamma\]
Let $M_\tau$ denote the quotient of $SL(2,\R)$ by the action $*_\tau$ and note that the left multiplication action of $g_t = \begin{pmatrix}
    e^t & 0 \\ 0 & e^{-t}
\end{pmatrix}$ commutes with the action of $*_\tau$. Hence, it descends to a flow on $M_\tau$ which we call a {\it canonical time change} of the geodesic flow on $SL(2,\R)/\Gamma$.

\begin{theorem}[\cite{G1,G2}]\label{theoreme ghys}
	Let $\Phi$ be a $C^\infty$ volume preserving Anosov flow on a $3$-manifold. 
    
    \begin{itemize}
        \item If the foliations $\mc W^{cu}_\Phi$ and $\mc W^{cs}_\Phi$ are $C^\infty$, then $\Phi$ is $C^\infty$-conjugated to a $C^\infty$ time change of a geodesic flow on a hyperbolic surface or the suspension of a linear Anosov diffeomorphism.
        \item If the foliations $\mathcal{W}_\Phi^{uu}$ and $\mathcal{W}_\Phi^{ss}$ are $C^\infty$, then $\Phi$ is $C^\infty$-conjugated to a canonical time change of a geodesic flow on a hyperbolic surface or a constant-time suspension of a linear Anosov diffeomorphism.
        \item If $\Phi$ is contact and $\mc W^{cu}_\Phi$ and $\mc W^{cs}_\Phi$ are $C^\infty$, then $\Phi$ is $C^\infty$-conjugated to a canonical time change of a geodesic flow on a hyperbolic surface.
        \item If $\Phi$ is the geodesic flow on a negatively curved surface and $\mc W^{cu}_\Phi$ and $\mc W^{cs}_\Phi$ are $C^\infty$, then the surface is hyperbolic.
    \end{itemize}
\end{theorem}

\subsection{Dynamical invariants from normal forms}\label{sec: data from normal forms}

The normal forms of Definition~\ref{def:resonance} developed in this paper capture a significant amount of information. Before presenting the proofs of their existence for vector fields, we first examine some of their key features. We show that from normal forms, one can get the Lyapunov exponents, the return time of the flow to the section, and the Anosov and Foulon-Hasselblatt classes around these periodic orbits. 

\pagebreak

\begin{proposition}
\label{prop:vectorfield data} 

Consider the resonance field on the solid torus $C=\R / \Z\times D_{\varepsilon}^2$:
\[ X =  f(xy) \left(\frac{\partial}{\partial t} + xg(xy) \frac{\partial}{\partial x} - yg(xy) \frac{\partial}{\partial y}\right) \]
for some $C^\infty$ functions $f,g\colon \R \to \R$ such that $f(0),g(0) > 0$. Then if $\varphi^t$ is the flow generated by $X$, 
\begin{enumerate}
    \item $\mathcal O:=\R / \Z \times \set{\vec 0}$ is a hyperbolic periodic orbit for the flow $\varphi^t$ with period $\frac{1}{f(0)}$.
        \item The numbers $\pm g(0)f(0)$ are the non-zero Lyapunov exponent of the flow $\varphi^t$ at the periodic orbit $\mathcal O$.
        \item The return time of the section $\set{0} \times D^2_\ve$ is the function $r(x,y) =\frac{1}{f(xy)}$. 
        \item 
        The first return map of the section $\set{0} \times D^2_\ve$ is the time-1 map of the vector field $$Y=g(xy)\left(x \frac{\del}{\del x} - y \frac{\del}{\del y}\right).$$ 
        \item $-g'(0)$ is the value of the Anosov class at the periodic orbit $\mathcal O$.
        \item $\frac{f'(0)}{f(0)}$ is the value of the Foulon-Hasselblatt class at the periodic orbit $\mathcal O$.
\end{enumerate}
    
\end{proposition}
\begin{proof}
Observe that $1/f\cdot X$ and $X$ locally induce the same Poincar\'{e} map on $D^2_\ve$, since $f$ is positive in a neighborhood. Since $1/f \cdot X = \frac{\del}{\del t} + x g \frac{\del}{\del x} - y g \frac{\del}{\del y}$, it follows that the first return map must be the time-1 map of $Y = g(xy) \left(x \frac{\del}{\del x} - y \frac{\del}{\del y}\right)$, which shows (4). If we denote by $\psi^t$ the flow generated by $Y$, $\psi^t$ can be solved explicitly as 
\begin{equation}\label{eq: flow in section}
\psi^t(x,y) = (e^{tg(xy)}x,e^{-tg(xy)}y).
\end{equation}

In fact, the flow for $X$ can be solved explicitly as well. Indeed, since the hypersurfaces $xy = \mbox{const}$ are invariant under $\psi^t$, the $\del / \del t$ component does not change along an orbit, so
\begin{equation}\label{eq: flow in coord}
 \varphi^t(s,x,y) = \left(s + f(xy)t,e^{tf(xy)g(xy)}x,e^{-tf(xy)g(xy)}y\right).
\end{equation}
 
It follows that the roof function is exactly $r = 1/f$, and the period of $\R / \Z \times \set{\vec 0}$ is $1/f(0)$, proving (1) and (3). (2) also follows by explicit calculation of the derivative of $\varphi^{1/f(0)}$. Finally, (5) and (6) follows from \eqref{eq: flow in coord} and the direct computation using Definitions \ref{def:anosov-cocycle} and \ref{def: Foulon-Hasselblatt class}, respectively.
\end{proof}

\begin{remark}
    By Theorem~\ref{introthm: vector field} and Proposition ~\ref{prop:vectorfield data}, the conditions that the Anosov class of any periodic orbit $\mathcal O$ in Theorem~\ref{theorem HK}, the Foulon-Hasselblatt class in Theorem~\ref{theorem FH} vanish, and the condition that the Lyapunov exponents for any periodic orbit coincide in \cite[Proposition 3.1]{DSLVY} can be replaced by the corresponding expressions of those invariants using the normal forms.  
\end{remark}

Using the resonance form of a contact form, we can improve Proposition~\ref{prop:vectorfield data} to the following in this setting.

\pagebreak

\begin{proposition}
\label{prop:cocycle-data}
    Let $\theta \in C^\infty(-\ve,\ve)$ be a positive function, $\alpha = \theta(xy)
    \, dt + \frac{1}{2}(x \, dy - y \, dx)$ be a contact form on $\R /\Z \times D^2_\ve$, and $\Phi$ be the local contact flow determined by $\alpha$. Then

    \begin{enumerate}
        \item $\mathcal O:=\R / \Z\times \{\vec 0\}$ is a hyperbolic periodic orbit for the flow $\Phi$ with period $\theta(0)$.
        \item The numbers $\pm\theta'(0)/\theta(0)$ are the non-zero Lyapunov exponents of the flow $\Phi$ at the periodic orbit $\mathcal O$.
        \item The return time of the section $\set{0} \times D^2_\ve$ is the function $r(x,y) =\bar{r}(xy)$, where $\bar{r}(z)=\theta(z) - z\theta'(z)$. 
        \item 
        The first return map of the section $\set{0} \times D^2_\ve$ is the time-1 map of the vector field $V_\theta = \theta'(xy)\left( x\frac{\del}{\del x} - y \frac{\del}{\del y}\right)$.
        \item The number $-\theta''(0)$ is the Anosov class at the periodic orbit $\mathcal O$.
        \item $\alpha \wedge d\alpha = \bar{r}\, dt \wedge dx \wedge dy$.
    \end{enumerate}
\end{proposition}
\begin{proof}
Let $\eta(z) = \theta'(z)$ and $r(x,y)$ be defined as in (3) with corresponding function of one variable $\bar r(z) = \theta(z) - z\theta'(z) = \theta(z) -z\eta(z)$. We claim that the vector field
\begin{equation}\label{normalcontactvector}
X_\theta = \frac{1}{\bar r(xy)}\frac{\del}{\del t} + \frac{\eta(xy)}{\bar r(xy)} \left(x \frac{\del}{\del x} - y \frac{\del}{\del y}\right)
\end{equation} 
is the Reeb field of $\alpha$. Indeed, observe that
\[ \alpha(X_\theta) = \theta(xy) \cdot \frac{1}{\bar r(xy)} - \frac{1}{2} \cdot 2xy \frac{\eta(xy)}{\bar r(xy)} = \dfrac{\theta(xy) - xy\eta(xy)}{\bar r(xy)} = \frac{\bar r(xy)}{\bar r(xy)} \equiv 1.\]
Furthermore, we compute $d\alpha$ explicitly as
\begin{align*}
d\alpha &= y\theta'(xy) \, dx \wedge dt + x \theta'(xy) \, dy \wedge dt  + dx \wedge dy\\ &= y\eta(xy) \, dx \wedge dt + x \eta(xy) \, dy \wedge dt + dx \wedge dy.
\end{align*}
Thus, 
\begin{eqnarray*}\iota_{X_\theta}d\alpha & = & \frac{1}{\bar r(xy)}\iota_{\del/\del t}\alpha + \frac{\eta(xy)}{\bar r(xy)} \cdot x \iota_{\del/\del x}\alpha - \frac{\eta(xy)}{\bar r(xy)}\cdot y\iota_{\del / \del y}\alpha \\
 & = & \left(-\dfrac{y\eta(xy)}{\bar r(xy)} dx - \dfrac{x\eta(xy)}{\bar r(xy)}\, dy\right) + \left(\frac{xy\eta^2(xy)}{r} dt +\frac{\eta(xy)}{\bar r(xy)} \cdot x\, dy\right) \\&-& \left(\frac{xy\eta^2(xy)}{\bar r(xy0}dt - \frac{\eta(xy)}{\bar r(xy)}\cdot y \, dx\right) \\
 & \equiv & 0.
\end{eqnarray*}
From expression \eqref{normalcontactvector} of the vector field $X_{\theta}$, we know that 
$$f(z)=\frac{1}{\theta(z)-z\theta'(z)} \text{ and } g(z)=\theta'(z).$$
Applying Proposition \ref{prop:vectorfield data} with the above expressions for $f$ and $g$, the properties (1)-(5) follow directly. 
\end{proof}

\begin{remark}
\label{rem:r-derivatives}
    Notice that the formula in (5) can be rewritten as $r(x,y) = \bar{r}(xy)$, with $\bar{r}(z) = \theta(z) - z\theta'(z)$. Furthermore, $\bar{r}'(z) = -z\theta''(z)$, so the Foulon-Hasselblatt cocycle $\bar{r}'(0)$ vanishes automatically. In fact this is the only derivative of $\bar{r}$ to have automatic vanishing at 0, $r''(0) = -\theta''(0)$ coincides with the Anosov class.
\end{remark}

\section{Normal form for the generator of a volume-preserving flow\\ in a neighborhood of a hyperbolic periodic orbits}\label{sec: normal forms for vector fields}
In this section, we prove Theorem~\ref{introthm: vector field}. A key idea is to establish a normal form for cocycles at hyperbolic periodic orbits as well as the return dynamics (Section \ref{sec: normal form for cocycles}), which is then used in Section \ref{sec: coordinates} to produce the normal forms charts.

\subsection{Normal forms for local cocycles around a hyperbolic fixed point}\label{sec: normal form for cocycles}

In this section, we will obtain the normal forms of local cocycles around hyperbolic fixed points of conservative local diffeomorphisms. We begin by defining the notion of a {\it local cocycle}.

\begin{definition}\label{def: cocycle}
    Let $f\colon D^2_\ve \to \R^2$ be a local diffeomorphism of $\R^2$ near $\vec 0$ such that $f(\vec 0)=\vec 0$, and $D \subset \Z \times D^2_\ve$ be the set $D = \set{(n,\vec x) : f^n(\vec x) \in D^2_\ve}$, so that $$D_n = \set{\vec x \in D^2_\ve : (k,\vec x) \in D\mbox{ for all }k \in[-n,n]}$$ is a shrinking neighborhood of $\vec 0$. A {\it local ($\R$-valued) cocycle} over $f$ is a function $\mc A\colon D \to \R$ such that $\mc A(m+n,\vec x) = \mc A(m,\vec x) + \mc A(n,f^m(\vec x))$ for all $m, n \in \Z$ and $\vec x \in D_{\abs{m}+\abs{n}}$.

    $\mc A$ is said to be a {\it local ($\R$-valued) $C^r$-coboundary} if there exists some $u \in C^r(D^2_\ve)$ such that $\mc A(n,\vec x) = u(f^n(\vec x)) - u(\vec x)$ for all $\vec x \in D_n$. Two cocycles are said to be cohomologous if they differ by a coboundary.
\end{definition}

As with the usual cocycle theory, each cocycle $\mc A$ has an associated generator $\phi = \mc A(1,\vec x)$, and $\mc A$ will be a local $C^r$-coboundary if and only if there exists $h \in C^r(D^2_\ve)$ such that $\phi = h \of f - h$ on $D_1$. We will use ``cocycle terminology" when we also talk about the associated generator. 

Note that many of the usual methods in cohomology theory fail. For instance, one may not use the usual relation $h(\vec y) - h(\vec x) = \lim\limits_{n \to \infty}\mc A(n,\vec x) - \mc A(n,\vec y)$ along stable manifolds since $\mc A$ is not defined for all times.

Let $F$ be a $C^{\infty}$ local area-preserving diffeomorphism of $\mathbb{R}^2$ with a hyperbolic fixed point $\vec 0$. By Theorem \ref{thm:sternberg}, there exist a coordinate system for which 
$$F(x,y)=(\lambda x\omega(xy), \lambda^{-1}y(\omega(xy))^{-1})$$
around $\vec 0$, where $\omega(z)$ 
is a smooth function such that $\omega(0)=1$ and $\lambda >1$.  
If $\mc A$ is a local cocycle over $F$, 
the following proposition tells us that up to a coboundary, $\phi$ can be regarded as a function of $xy$ (and that representative in local cohomology is unique). 

We first make a terminological clarification. A multi-index $\vec k$ is a pair $(i,j) \in \N_0$, with order $\abs{\vec k} = i+j$. If $\phi \in C^\infty(D^2_\ve)$ and $\vec k = (i,j)$, $\phi^{(\vec k)}(x)$ is the partial derivative

\[ \phi^{(\vec k)}(x,y) = \del_x^i\del_y^j\phi(x,y)\]

\begin{proposition}\label{nfc}
Let $F(x,y)=(\lambda x\omega(xy), \lambda^{-1}y(\omega(xy))^{-1})$ on a neighborhood of $\vec 0$ in $\mathbb R^2$ where $\lambda>1$ and $\omega(z)$ is a smooth function such that $\omega(0)=1$. Denote by $\mc A$ a $C^\infty$ local cocycle over $F$ generated by a function $\phi$. Then there exists $\bar \phi \in C^\infty(-\ve,\ve)$ for some $\ve$ such that $\phi$ is cohomologous to $\tilde{\phi}(x,y)= \bar\phi(xy)$. Furthermore, if $\phi^{(\vec k)}(\vec 0) =0$ for all $\vec k\in\mathbb N_0^2$ such that $\abs{\vec k} \le k_0$ for some $k_0\in\mathbb N_0$, then $\bar\phi^{(\ell)}(0) = 0$ for all $\ell \le k_0$.
\end{proposition}

First, we prove Proposition~\ref{nfc} on the level of formal power series. For emphasis, when working with formal power series, we write $\approx$ rather than equality to indicate that two formal power series are the same (or that two functions have the same formal power series).

\begin{lemma}
\label{lem:formal-sol}(cf. Proposition 6.6.1 in \cite{KH})
If $\phi \approx \displaystyle\sum_{\substack{k,\ell =0 \\ k\not=\ell}}^\infty c_{k,\ell}x^ky^\ell$ is a formal power series, and $F$ is as in Proposition \ref{nfc} has formal power series approximation

\[ F(x,y)\approx \left(\lambda\sum_{i=0}^{\infty}a_{i}x^{i+1}y^i, \lambda^{-1}\sum_{i=0}^{\infty}b_{i}x^iy^{i+1}\right), \qquad a_0=b_0=1,\] then there exists a formal power series $w \approx \displaystyle\sum_{\substack{i,j = 0\\i\not=j}}^\infty w_{i,j}x^iy^j$ such that $w \of f - w \approx \phi$.
\end{lemma}

\begin{proof}
Consider $w \approx \displaystyle\sum_{i,j = 0}^\infty w_{i,j}x^iy^j$.
We directly compute
\begin{equation}
\label{eq:coboundary}
w\circ F(x,y)-w(x,y)\approx \sum_{i,j=0}^{\infty}w_{i,j}\lambda^{i-j}\left(\sum_{m=0}^{\infty}a_{m}x^{m+1}y^m\right)^i\left(\sum_{n=0}^{\infty}b_{n}x^ny^{n+1}\right)^j-\sum_{i,j=0}^{\infty}w_{i,j}x^iy^j. 
\end{equation}
For any given $(k,\ell)$, the coefficient of $x^ky^\ell$ in the expression above is computed directly as a polynomial of the coefficients of each component whenever $i(m+1) + jn =k$ and $im + j(n+1) = \ell$ (as well as the $-w_{k,\ell}x^ky^\ell$ term). 
It follows that the coefficients of $x^ky^\ell$ only depends on $w_{k',\ell'}$ with $k' \le k$ and $\ell' \le \ell$, with the $w_{k,\ell}$ appearing linearly as $(\lambda^{k-\ell}-1)w_{k,\ell}$.
 
It follows that if $k\not=\ell$ we can solve
$$w_{k,\ell}=\frac{v_{k,\ell}}{(\lambda^{k-\ell}-1)},$$
where $v_{k,\ell}$ is a polynomial in the variables $w_{k',\ell'}$, $c_{k,\ell}$, $a_m$ and $b_n$.
Hence, by induction, any power series without resonance terms $x^ky^k$ can be expressed as a coboundary.
\end{proof}

\begin{lemma}\label{Taylor}(cf. Proposition 6.6.3 in \cite{KH}) For any sequence $(a_{i,j})_{i,j\in\mathbb{N}_0}$ of numbers, there exists a $C^{\infty}$ function $\phi: \mathbb{R}^2\rightarrow \mathbb{R}$ such that $a_{i,j}$ are the Taylor coefficients of $\phi$ at $0$.
\end{lemma}
\begin{proof}
 Consider $\phi(x,y)=\sum^{\infty}_{i=0,j=0}a_{i,j}x^iy^jb(|i+j|!C_{i+j}(x^2+y^2)), $
 where $b(t):\mathbb{R}\rightarrow \mathbb{R}$ is a bump function with value $1$ on $[-1,1]$ and $C_{N}=\sum_{l=0}^N\sum_{i+j=l}|a_{i,j}|$. Notice that this series converges since for each $x\neq 0$ there are only finitely many non-zero terms and the sum converges uniformly and very rapidly. Moreover, the derivatives of order $N$ around $0$ satisfy what we want. 
\end{proof}

Next lemma allows us to show that if we can solve a cocycle equation at the level of formal power series, then we can find an actual solution.

\begin{lemma}\label{global}(cf. Proposition 6.6.5 in \cite{KH})Let $F$ be as in Proposition \ref{nfc}, and 
$\phi$ be a local $C^{\infty}$ function such that $\phi^{(\vec k)}(\vec 0)=0$ for all $\vec k\in\mathbb N_0^2$. 
Then there is a $C^{\infty}$ function $v$ defined on a neighborhood of $\vec 0$ such that $v\circ F-v=\phi$. 
\end{lemma} 
\begin{proof}
First, we take arbitrary extensions of $\phi$ to $\mathbb R^2$ and $\omega$ to $\mathbb R$ which extends $F$ to $\mathbb R^2$. Below, we work with the extensions that we still denote by $\phi, \omega$, and $F$.

Then, we decompose $\phi$ as 
$\phi=\phi^{+}+\phi^{-}$
such that $\phi^{+}$ and all its jets vanish along the $x$-axis and $\phi^{-}$ and all its jets vanishes along $y$-axis.  To construct this decomposition, we can take a smooth function $\rho$ on the unit circle $S^1$ such that $\rho\equiv 0$ on the intersection of $S^1$ and the horizontal cone $C_H:=\{(x,y) : |y|\leq\frac{1}{2}|x|\}$ and $\rho\equiv1$ on the intersection of $S^1$ and the vertical cone $C_V:=\{(x,y) : |x|\leq\frac{1}{2}|y|\}$. We set
$$\phi^{+}(x,y)=\phi(x,y)\rho\left(\frac{x}{\sqrt{x^2+y^2}},\frac{y}{\sqrt{x^2+y^2}}\right), \text{ for }x\neq 0$$
and $\phi^{+}(0)=0$. Moreover, we define $\phi^{-}=\phi-\phi^{+}$. We note that the constructed functions are $C^\infty$ at $\vec 0$ and have the desired properties. Indeed, if $(x,y)\neq \vec 0$, then $(\phi^+)^{(\vec k)}(x,y)$ for $\vec k\in\mathbb N_0^2$ is a polynomial in the derivatives of $\phi$, $\rho$, and $(x^2+y^2)^{-1/2}$, and each monomial contains $\phi$ or some of its derivative. Moreover, $\|\phi^{(\vec k)}(x,y)\|=o((x^2+y^2)^{m/2})$ for all $\vec k\in\mathbb N_0^2$ and $m\in\mathbb N$  as $\phi^{(\vec k)}(\vec 0)=0$ for all $\vec k\in\mathbb N_0^2$, and the derivatives of $\rho$ are bounded. As a result, $\|(\phi^+)^{(\vec k)}(x,y)\|=o((x^2+y^2)^{m/2})$ for all $\vec k\in\mathbb N_0^2$ and $m\in\mathbb N$ and hence $\phi^+$ is a $C^\infty$ function on a neighborhood of $\vec 0$ and so is $\phi^-$. 

We want to find two $C^\infty$ functions $v^{+}$ and $v^{-}$ such that 
\begin{equation}\label{eq:cocycle split}
v^{+}\circ F-v^{+}=\phi^{+} \text{ and }v^{-}\circ F-v^{-}=\phi^{-}.
\end{equation}

Then, $v=v^++v^-$ is a solution of $v\circ F-c=\phi$.

We note that $v^+=-\sum\limits_{ m=0}^\infty \phi^+\circ F^{m}$ and $v^-=\sum\limits_{ m=0}^\infty \phi^-\circ F^{-m-1}$ are formal solutions of \eqref{eq:cocycle split}. Moreover, if we show that the series converge in the $C^\infty$ topology then it would imply that we constructed desired solutions. 

We have $F^m(x,y)=(\lambda^mx(\omega(xy))^m,\lambda^{-m}y(\omega(xy))^{-m})$ for $m\in\mathbb Z$. Also, $(\phi^+\circ F^{m})^{(\vec k)}(x,y)$ is a polynomial in the derivatives of $\phi^+$ and $F^m$ of order up to $|\vec k|$, and each term contains a derivative of $\phi^+$ or $\phi^+$ itself, evaluated at $F^m(x,y)$ which exponentially converges to the $x$-axis as $m\rightarrow +\infty$. Thus, $(\phi^+\circ F^{m})^{(\vec k)}(x,y)$ superexponentially converges to $0$ as $m\rightarrow +\infty$ by the construction of $\phi^+$. Similarly, since  $F^{-m-1}(x,y)$ exponentially converges to the $y$-axis as $m\rightarrow +\infty$, $(\phi^-\circ F^{-m-1})^{(\vec k)}(x,y)$ superexponentially converges to $0$ as $m\rightarrow +\infty$ by the construction of $\phi^-$. Thus, $v^+$ and $v^-$ are $C^\infty$ functions, and we proved the lemma. 
\end{proof}

Now we are ready to finish the proof of Proposition \ref{nfc}.
\begin{proof}[Proof of Proposition \ref{nfc}]
Fix the local cocycle $\phi$ with Taylor expansion $\phi \approx \sum_{k,\ell=0}^\infty \phi_{k,\ell}x^ky^\ell$, and pick any $C^\infty$ function $\bar \phi$ such that the Taylor expansion of $\bar\phi$ is given by $\bar\phi(z) \approx \phi_{k,k}z^k$ (see Lemma~\ref{Taylor}). Then let $\phi_0 = \phi - \tilde{\phi}$, where $\tilde{\phi}(x,y) = \bar\phi(xy)$. By construction, $\phi_0$ has Taylor expansion $\displaystyle\phi_0 \approx \sum_{\substack{k,\ell=0\\k\not=\ell}}^\infty\phi_{k,\ell}x^ky^\ell$.

By Lemma \ref{lem:formal-sol}, there exists a formal power series $w$ such that $w \of F - w \approx \phi_0$. Moreover, by Lemma \ref{Taylor}, there exists a smooth function $\tilde{w}(x,y)$ with $w$ as its Taylor series. It follows that $\phi_0 - (\tilde w \of F - \tilde w)$ is $C^\infty$ tangent to $0$ at $\vec 0$, and is thus a coboundary by Lemma \ref{global}. In particular, $\phi_0 - (\tilde w \of F - \tilde w) = v \of F - v$ for some $C^\infty$ function $v$. Thus, $\phi_0$ is a coboundary with transfer function $v + \tilde w$, and $\phi$ is cohomologous to $\tilde{\phi}$, as claimed. 
\end{proof}

Finally, we show that the coefficients of the formal power series $\bar\phi(z)\approx \sum_{k=0}^\infty \bar\phi_kz^k$ from Proposition~\ref{nfc} are invariants of cohomology.

\begin{lemma}
\label{lem:cocycle-coefficeints}
    If $\phi \approx \sum_{k,\ell =0}^\infty \phi_{k,\ell} x^ky^\ell$ and $\phi' \approx \sum_{k,\ell =0}^\infty \phi_{k,\ell}'x^ky^\ell$ are cohomologous local cocycles over $F$ as in Proposition \ref{nfc}, then $\phi_{k,k} = \phi_{k,k}'$ for all $k \ge 0$.
\end{lemma}

\begin{proof}
    It suffices to show that any coboundary has vanishing diagonal terms. We verify this by rewriting \eqref{eq:coboundary} with additional information from $F$. Indeed, recall that we may instead write $F$ as
    \[F(x,y) \approx \left(\lambda x \omega(xy),\lambda^{-1}y(\omega(xy))^{-1}\right).\]

    In this case, $u \of F - u \approx \sum_{i,j=1}^\infty u_{i,j}x^iy^j((\lambda\omega(xy))^{i-j}-1)$. We may then replace $\omega(xy)^{i-j}$ with a power series depending only on $xy$. In particular, since the power series in $xy$ is multiplied by $x^iy^j$ the only candidate for a term of the form $x^ky^k$ must come from a summand with $i = j$. But if $i=j$, $(\lambda\omega(xy))^{i-j} - 1 = 0$. It follows that any coboundary has no diagonal terms in its formal Taylor approximation. 
\end{proof}

\subsection{Coordinates in a neighborhood of a hyperbolic periodic orbit}\label{sec: coordinates}
In this section, we complete the proof of Theorem \ref{introthm: vector field}. 
Let $p$ be a hyperbolic periodic point of a smooth flow $\Phi=(\varphi^t)_{t\in\mathbb R}\colon M\rightarrow M$ and $\Sigma\ni p$ be a Poincar\'e section with local coordinates $h_\Sigma\colon D^2_\ve \to \Sigma$ such that the Poincar\'e map has the form
$F(x,y)=(e^{g(xy)}x,e^{-g(xy)}y)$ for some $C^\infty$ function $g$ on a neighborhood of $0$ such that $g(0)>0$ (see Theorem \ref{thm:sternberg}). By Proposition \ref{nfc}, the return time $r$ to the section $\Sigma$ is cohomologous to a function $\bar r(xy)$ on $D_\varepsilon^2$, i.e., there exists $u\in C^\infty(D_\varepsilon^2)$ such that $u(\vec 0) = 0$ and
\begin{equation*}
r-\bar r= u\circ F-u.
\end{equation*}

 Consider a section $\Sigma_u = \{\varphi^{-u(x,y)}(h_\Sigma(x,y))\, : \, (x,y)\in D^2_\varepsilon\}$, and note that $\tilde{\Sigma}$ still passes through $p$ since $u(\vec 0) = 0$. By construction, the return time to $\tilde{\Sigma}$ is $\bar{r}$.

Let $Y_0 = x \frac{\del}{\del x}- y \frac{\del}{\del y}$. We, henceforth, assume that we have chosen our section $\Sigma$ to be the image of a function $h_\Sigma\colon D^2_\ve \to M$ and have the following properties: 
\begin{enumerate}
    \item In the coordinates provided by $h_\Sigma$, the first return map $F$ is the time-1 map of a flow $\psi^t\colon D^2_\ve\rightarrow \mathbb R^2$ induced by a vector field $Y = g(xy)Y_0$;
    \item $g$ is a $C^\infty$ function on a neighborhood of $0$;
    \item the return time function $r\colon\Sigma\rightarrow\mathbb R_+$ in the local coordinates given by $h_\Sigma$ is a $C^\infty$ function of $xy$. 
\end{enumerate}

We now define the function $H\colon \mathbb R/\mathbb Z\times D^2_{\ve'}\rightarrow M$ for some $\ve'>0$. 
Let $\ve'>0$ be such that if $(x,y) \in D^2_{\ve'}$, $\psi^{t}(x,y) \in D^2_\ve$ for $|t|\leq 1$. Then, we define
\[ H(t,x,y) = \varphi^{t\bar{r}(xy)}\left(h_{\Sigma}(\psi^{-t}(x,y))\right). \]
We claim that $H$ is well-defined, i.e., $H(t + 1,x,y) = H(t,x,y)$ whenever both points are defined. Indeed, notice that
\begin{equation}\label{eq: periodic H}
 H(t+1,x,y) = \varphi^{(t+1)\bar{r}(xy)}h_{\Sigma}(\psi^{-t-1}(x,y)) = \varphi^{t\bar{r}(xy)} \varphi^{\bar{r}(xy)} h_{\Sigma}(\psi^{-1}\psi^{-t}(x,y)) = H(t,x,y),
\end{equation}
since by definition, $\psi^{-1} = F^{-1}$ is the first return map for $\Phi$, $\bar{r}$ is the roof function, and $\psi^t$ preserves hyperbolas $xy=\mathrm{const}$ (see \eqref{eq: flow in section}). Let $f(xy) = \frac{1}{\bar{r}(xy)}$. Finally, we verify that the vector field of $\Phi$ in the coordinates given by $H$ is
 \[  \tilde{X} = f(xy)\left( \frac{\partial}{\partial t} + xg(xy) \frac{\partial}{\partial x} - yg(xy) \frac{\partial}{\partial y}\right).\]

The integral curves $\gamma_{s,x,y}(t)$ of $\tilde X$ are given by \eqref{eq: flow in coord}, and the orbits of $\psi^t$ are given by \eqref{eq: flow in section}.
Furthermore, note that
\begin{eqnarray*} H(\gamma_{s,x,y}(t)) & = & H(s+f(xy)t,e^{tf(xy)g(xy)}x,e^{-tf(xy)g(xy)}y)\\
 &= & \varphi^{(s+f(xy)t)\bar{r}(xy)}h_{\Sigma}(e^{-(s+f(xy)t)g(xy)}e^{tf(xy)g(xy)}x,e^{(s+f(xy)t)g(xy)}e^{-tf(xy)g(xy)}y)
 \\
 &= & \varphi^{t+s\bar{r}(xy)}h_{\Sigma}(e^{-sg(xy)}x,e^{sg(xy)}y) 
 \\
 & = & \varphi^t(H(s,x,y))
\end{eqnarray*}
Thus, $H$ intertwines the flows generated by $\tilde{X}$ and $X$ on $\mathbb R/\mathbb Z\times D^2_{\ve'}$ and $M$, respectively.

\section{Normal form for a contact form in a neighborhood\\ of a hyperbolic periodic orbit of the Reeb flow}\label{sec:normal form for contact}
In this section, we present the proof of Theorem \ref{thm:contact-normal}. Unlike the approach in Section \ref{sec: normal forms for vector fields}, where we worked primarily with the vector field, here we work directly with the contact form. 
As a byproduct, improved expressions for the associated contact vector fields around the tubular neighborhood of periodic orbits are obtained for Anosov Reeb flows on manifolds of dimension $3$.  

\subsection{Choice of a section}
\label{sec:contact-setup}
Consider a contact form $\alpha$, and denote by $\Phi = (\varphi^t)_{t \in \R}$ the associated Reeb flow. Let $\mathcal O$ be a hyperbolic periodic orbit of $\Phi$, and let $p\in \mathcal O$. In this section, we prove the following lemma.

\begin{lemma}\label{lemma: nice section}
   There exists a section $\Sigma$  transverse to $\Phi$  at $p$
with local coordinates $h\colon D^2_\ve \to \Sigma$ for some $\ve>0$ such that 
\begin{enumerate}
    \item $h^*\alpha = \frac{1}{2}(x\, dy - y \, dx)$
    \item $h^*d\alpha = dx \wedge dy$
    \item the Poincar\'e map has the form
$F(x,y)=(e^{\eta(xy)}x,e^{-\eta(xy)}y)$ for some $C^\infty$ function $\eta$ on a neighborhood of $0$ such that $\eta(0)>0$. In particular, $F$ is the time-1 map of a flow $\psi^t$ induced by a vector field $Y = \eta(xy)Y_0$.
\end{enumerate}
\end{lemma}

\begin{proof}
Choose a section $\Sigma'$ transverse to $\Phi$  at $p$
with local coordinates $h'\colon D^2_\ve \to \Sigma$ for some $\ve>0$ such that the Poincar\'e map has the form
$F'(x,y)=(e^{\eta(xy)}x,e^{-\eta(xy)}y)$ for some $C^\infty$ function $\eta$ on a neighborhood of $0$ such that $\eta(0)>0$ (see Theorem \ref{thm:sternberg}). Since $\Phi$ preserves the contact volume, $F'$ preserves the area form $(h')^*d\alpha=dx\wedge dy$ restricted to $\Sigma$.

Now we modify the section $\Sigma'$. For any $C^\infty$ function $\tau\colon D^2_\ve \to \R$, we define  a $C^\infty$ function $h_\tau\colon D^2_\ve \to X$ given by $h_\tau(x,y) = \varphi^{\tau(x,y)}(h'(x,y)).$ Then, to each such $\tau$ we can associate a section $\Sigma_\tau$ built from our initial choice $\Sigma$ in the following
\begin{equation}\label{eq: sigma_tau}
 \Sigma_\tau = \set{ h_\tau(x,y) : (x,y) \in D^2_\ve}.
\end{equation}

\begin{lemma}
\label{lem:canonnical-section}
    There exists a $C^\infty$ function $\tau$ such that $h_\tau^*\alpha = \frac{1}{2}(x\, dy - y \, dx)$. 
\end{lemma}
\begin{proof}
    Given a vector $V$ on $D^2_\ve$, let $\gamma\colon (-\ve,\ve) \to D^2_\ve$ be a curve such that $\gamma'(0) = V$. Then,
    \[ dh_\tau(V) = \left.\frac{d}{dt}\right|_{t=0} \varphi^{\tau(\gamma(t))}(h'(\gamma(t)).\]
  Differentiating in $t$ yields $d\tau(V) \cdot X(h_\tau(\gamma(0)) + d\varphi^{\tau(\gamma(0))}(dh'(V))$. By using invariance of $\alpha$ under the contact flow, we can evaluate $\alpha$ on this vector field to obtain 
    \[ h_\tau^*\alpha = d\tau + (h')^*\alpha. \]
 Note that since $\ker \alpha$ is not integrable, $\alpha$ cannot vanish identically on $\Sigma$. Given our initial section map $h'$, we may write $(h')^*\alpha = (a-\frac{1}{2}y) \, dx + (b+\frac{1}{2}x) \, dy$ for some $C^\infty$ functions $a$ and $b$ on $D^2_\ve$. Note that since $d(h')^*\alpha=(h')^*d\alpha =dx \wedge dy$, we get that $\partial_xb = \partial_ya$. Therefore, the form $a \, dx + b \, dy$ is closed. Since the first de Rham cohomology group $H^1_{dR}(D^2_\ve) = 0$, there exists a $C^\infty$ function $\tau$ such that $d\tau = -a \, dx - b\, dy$. Such a $\tau$ is a desired solution.
\end{proof}

By shrinking $\ve$ as necessary, $h_\tau$ will be a diffeomorphism onto its image. We choose $\Sigma$ as the image of a function $h_\tau$ where $\tau$ from Lemma~\ref{lem:canonnical-section}.
\end{proof}

\subsection{Properties of the return time}
\label{sec:contact-cocycle}
Let $\Sigma$ and $h$ be as in Lemma~\ref{lemma: nice section}, and let $r\colon D^2_\ve \to \R_+$ be the associated return time for the flow $\Phi$ (see Section~\ref{sec: sections}) written in the coordinates given by $h$. 
\begin{lemma}
\label{lem:dr-formula}
    $dr = -xy \,d\eta$ where $\eta$ as in Lemma~\ref{lemma: nice section}. 
    
    In particular, $r$ is a $C^\infty$ function of $xy$ given by
 \[r(x,y) - c =  \int_0^{xy} \eta(s) \, ds -  xy\, \eta(xy).\]
\end{lemma}
\begin{proof}
We compute $dr$ by choosing an arbitrary curve $\gamma\colon [-\ve,\ve] \to D^2_\ve$, and evaluating $dr(\gamma'(0))=(r \of \gamma)'(0)$. We will show that $(r \of \gamma)'(0)=-\gamma_1(0)\gamma_2(0)(\eta\of \gamma)'(0)$, where $\gamma_i$ is the $i$th component of $\gamma$.

Let $\gamma$ be such a curve, and consider the surface with boundary
\[ S = \set{ \varphi^t(h(\gamma(s))) : s \in [-\ve,\ve],t \in [0,r(\gamma(s))]}\]
Note that $S$ is a rectangle, and is everywhere tangent to the Reeb vector field $X$. Since $\iota_Xd\alpha \equiv 0$, we conclude that $\alpha$ is closed when restricted to $S$.

For each $s \in [-\ve,\ve]$, let $\nu_s$ denote the orbit segment of $\Phi$ in $S$ starting from $h(\gamma(s))$ and ending at $\varphi^{r(h(\gamma(s)))}(h(\gamma(s))) = h(F(\gamma(s)))$. Then $r(\gamma(s)) = \int_{\nu_s} \alpha$, since $\alpha(X) \equiv 1$, and one may use the Reeb flow to parameterize $\nu_s$.

Similarly, for each such $s$, let $\gamma_s$ denote the restriction of the curve $\gamma$ to the domain $[0,s]$. Then, Stoke's theorem implies that
\begin{equation}
\label{eq:r-derivative}
r(\gamma(s)) - r(\gamma(0)) = \int_{\nu_s} \alpha - \int_{\nu_0} \alpha = \int_{F \of \gamma_s} \alpha - \int_{\gamma_s} \alpha.
\end{equation}
Let $\beta = h^*\alpha = \frac{1}{2}(x \, dy - y \, dx)$ for simplicity of notation. Dividing \eqref{eq:r-derivative} by $s$ and taking the limit as $s \to 0$ yields that $dr = F^*\beta - \beta$. Now, since $F$ is the time-1 map of $\psi^t$, it follows that
\begin{equation}\label{eq:Tstar-beta} F^*\beta = (\psi^1)^*\beta = \beta + \int_0^1 \left.\dfrac{d}{dt}\right|_{t=s}(\psi^t)^*\beta \, ds= \beta +\int_0^1 (\psi^s)^*\mc L_Y\beta \, ds
\end{equation}
Observe that $\beta$ is invariant under the flow generated by $Y_0$, and $\beta(Y_0) \equiv -xy$, so that
\[ \mc L_Y\beta = \mc L_{\eta Y_0}\beta = \eta \, \mc L_{Y_0}\beta + d\eta \wedge \iota_{Y_0}\beta = -xy \, d\eta \]
Furthermore, 
\begin{equation}\label{eq:eta-Y-inv} \mc L_Y(-xy\, d\eta) = \mc L_Y(-xy) d \eta - xy \mc L_Yd\eta = 0 - xy\, d\mc L_Y\eta = 0
\end{equation}
This follows since the derivative of 
a function of $xy$ along the vector field $Y$ is $0$ as the integral curves of $Y$ are the hyperbolas $xy = \mathrm{const}$. Therefore, the form $-xy\, d\eta$ is invariant under the flow $\psi^s$, and $(\psi^s)^*\mc L_Y\beta = -xy\, d\eta$ for all $s$. Then, \eqref{eq:Tstar-beta} becomes 
\[ dr = F^*\beta - \beta = (\beta - xy \, d\eta) - \beta = -xy \, d\eta.\]
The integral formula for $r$ follows immediately from integration by parts.
\end{proof}

\subsection{Coordinates in a neighborhood of a hyperbolic periodic orbit of a Reeb flow}
\label{sec:contact-build}

In this section, we finish the proof of Theorem~\ref{thm:contact-normal}.

By Lemma~\ref{lem:dr-formula}, there exists a unique $C^\infty$ function $\bar{r} : \R \to \R$ such that $r(x,y) = \bar{r}(xy)$ and $\bar{r}'(z) = -z \eta'(z)$. In particular, for some $c>0$, we have
\[ \bar{r}(z) - c = \int_0^z \eta(s) \, ds - z \eta(z). \]
We define the function $H\colon \mathbb R/\mathbb Z\times D^2_\ve\rightarrow M$ for some $\ve>0$ by
\[ H(t,x,y) = \varphi^{t\bar{r}(xy)}(h(\psi^{-t}(x,y))). \]
We note that $H$ is well defined (see \eqref{eq: periodic H}).

\begin{lemma}
$H^*\alpha = (\bar{r}(xy) + xy\, \eta(xy))dt+\frac{1}{2}(x\,dy-y\,dx).$
\end{lemma}
\begin{proof}
Since $H^*\alpha$ is a 1-form on $\R / \Z \times D^2_\ve$, it must be a linear combination of $dt$, $dx$ and $dy$ at every point. It therefore suffices to evaluate $\alpha$ on $H_*(\del / \del t)$, $H_*(\del / \del x)$ and $H_*(\del / \del y)$ to compute these coefficients.

First, we compute the $dt$ coefficient. Fix a point $(t,x,y)$ and consider the curve $s \mapsto (t+s,x,y)$ in $\R / \Z \times D^2_\ve$. Then $H_*(\del / \del t)$ is equal to
\begin{multline*} \left.\frac{d}{ds}\right|_{s= 0} H(t+s,x,y) = \left.\frac{d}{ds}\right|_{s= 0} \varphi^{t\bar{r}(xy) + s\bar{r}(xy)}h \psi^{-s-t}(x,y) \\= \bar{r}(xy)X(H(t,x,y)) - d\varphi^{t\bar{r}(xy)} dh (Y(\psi^{-t}(x,y))).
\end{multline*}
Hence,
\begin{multline*}
    \alpha(H_*(\del / \del t)) = \bar{r}(xy) - \frac{1}{2}(x \, dy - y \, dx)\left[\eta(xy)\left(x \frac{\del}{\del x} - y \frac{\del}{\del y}\right)\right] = \bar{r}(xy) + xy\, \eta(xy).
\end{multline*}
Now we compute the $dx$ and $dy$ coefficients. As in the proof of Lemma \ref{lem:dr-formula}, we let $\beta = h^*\alpha = \frac{1}{2}(x \, dy - y \, dx)$. We wish to show that $dx$ and $dy$ coefficients are always $-\frac{1}{2}y$ and $\frac{1}{2}x$, respectively. We will do this by showing they are independent of $t$, and computing them when $t = 0$.

Fix $v \in \R^2$ and consider the curve $s\mapsto (t,((x,y)+sv)$ in $\mathbb R/\mathbb Z\times D^2_\varepsilon$. Then,
\begin{multline*}\left.\frac{d}{ds}\right|_{s= 0} H(t,(x,y)+sv) = \left.\frac{d}{ds}\right|_{s= 0} \varphi^{tr((x,y)+sv)}h \psi^{-t}((x,y)+sv) \\
= tdr_{(x,y)}(v) X(H(t,x,y)) + d\varphi^{t\bar{r}(xy)}dhd\psi^{-t}(v).
\end{multline*} 
Here, we use $dr_{(x,y)}$ to emphasize that we are evaluating the form $dr$ at vector $v$ based at the point $(x,y)$. Evaluating $\alpha$ on $\left. \frac{d}{ds}\right|_{s= 0} H(t,(x,y)+sv)$ yields the function
\[ tdr_{(x,y)}(v) + (\psi^{-t})^*h^*\alpha_{(x,y)}(v) = t dr_{(x,y)}(v) + (\psi^{-t})^*\beta_{(x,y)}. \]
Taking the derivative with respect to $t$ yields
\[ dr_{(x,y)}(v) {\color{red}-}(\psi^{-t})^*\mc L_Y\beta_{(x,y)} = dr_{(x,y)}(v) - (\psi^{-t})^*(-xy \, d\eta)(v)\]
Since $(\psi^{-t})^*(-xy \, d\eta) = -xy\, d\eta$ by \eqref{eq:eta-Y-inv}, we conclude from Lemma \ref{lem:dr-formula} that the derivative in the $t$-coordinate is 0. Hence, we may compute the form when $t= 0$, so $\alpha_{(t,x,y)}(v) = \beta_{(x,y)}(v)$ for all $(t,x,y)$.
\end{proof}

\subsection{Rigidity for Anosov contact flows}\label{sec:rigidity}
In this section, we prove Corollaries~\ref{coro C}~and~\ref{coro D}.

\begin{proof}[The proof of Corollary \ref{coro C}]
      Assume that such sections $\Sigma_p$ exist at every periodic orbit $p$. This implies that all derivatives of the roof function $r$ up to order four vanish. Since the resonance terms of the roof function $r$ cannot be modified through cohomology (Lemma \ref{lem:cocycle-coefficeints}), this implies that for the modified section constructed in Lemma~\ref{lemma: nice section}, $\bar{r}''(0) = \del_x^2\del_y^2r(\vec 0) = 0$.
    The result is now immediate from Remark \ref{rem:r-derivatives}, Proposition \ref{prop:cocycle-data}, Theorem \ref{theorem HK} and Theorem~\ref{theoreme ghys}.
\end{proof}

\begin{proof}[The proof of Corollary \ref{coro D}]
    If $S$ has constant curvature, one can find sections $\Sigma_\gamma$ such that $\varphi^T\Sigma_\gamma \cap \Sigma_\gamma$ is an open subset of\ both $\Sigma_\gamma$ and $\varphi^T\Sigma_\gamma$, as computed in Example \ref{ex:homogeneous}. The other direction follows immediately from Corollary \ref{coro C} and Theorem \ref{theoreme ghys}. 
\end{proof}

\section{Base-Roof Rigidity for contact flows}\label{sec: base-roof rigidity}

In this section we prove Theorems \ref{thm E} and \ref{thm:linear-rigidity}, which establish a new rigidity phenomenon for contact flows in which the dynamics of the return map determine properties of the return time, and vice-versa. We propose to call this feature ``base-roof rigidity,'' terminology which appears from thinking about suspension flows.

\subsection{The proof of Theorem \ref{thm E}}
Inspired by the key formula built in Lemma \ref{lem:dr-formula},
we investigate the connection among the roof
function, the return maps and contact forms further.  In this subsection, we give the proof of Theorem \ref{thm E}. 

\begin{lemma}
\label{lem:contact-uniqueness}
    Let $\alpha_0$ and $\alpha_1$ be two contact forms on $C = \R / \Z \times D^2_\ve$, and assume that

    \begin{itemize}
        \item $\R / \Z \times \set{0}$ is a periodic orbit for both $\alpha_0$ and $\alpha_1$,
        \item the Reeb flows induced by $\alpha_0$ and $\alpha_1$ coincide, and
        \item $\alpha_0 \wedge d\alpha_0 = \alpha_1 \wedge d\alpha_1$.
    \end{itemize}

    Then $\alpha_0 = \Psi^*\alpha_1$ for some $C^\infty$ diffeomorphism $\Psi$.
\end{lemma}

\begin{proof}
    Let $X$ denote the common Reeb field for $\alpha_0$ and $\alpha_1$ and $\beta = \alpha_1 - \alpha_0$. Then since $\alpha_i(X) \equiv 1$ and $X \in \ker d\alpha_i$ for $i=1,2$, we get that $X \in \ker \beta$ and $X \in \ker d\beta$. Define $\alpha_t = \alpha_0 + t\beta$, and note that this notation is consistent with our initial indexing of $\alpha_0$ and $\alpha_1$.

    We apply the Moser trick to produce the conjugacy $\Psi$. Indeed, we will produce this by finding a $C^\infty$ function $f$ such that if $V = fX$, and $\Psi^t$ is the time-$t$ map of the flow generated by $V$, then $(\Psi^t)^*\alpha_t = \alpha_0$. Indeed, this occurs if and only if $\frac{d}{dt}(\Psi^t)^*\alpha_t \equiv 0$. We compute
    \begin{multline*}
         \frac{d}{dt}(\Psi^t)^*\alpha_t = (\Psi^t)^*(\mc L_V\alpha_t + \frac{d}{dt}\alpha_t) = (\Psi^t)^*(d\iota_V\alpha_t + \iota_Vd\alpha_t + \beta) \\= (\Psi^t)^*(d\iota_{fX}(\alpha_0+t\beta) + \iota_{fX}d(\alpha_0+t\beta) + \beta) \\ =(\Psi^t)^*( d(f\cdot 1 + 0) + 0 + \beta) = (\Psi^t)^*(df + \beta).
    \end{multline*}
    Therefore, we may pick $f$ to make this identically 0
 if and only if $\beta$ is exact. But notice that since $\alpha_0\wedge d\alpha_0 = \alpha_1 \wedge d\alpha_1$,
 \[0 = \iota_{X}(\alpha_1 \wedge d\alpha_1 - \alpha_0 \wedge d\alpha_0) = d\alpha_1 - d\alpha_0 = d\beta.\]
Therefore, $\beta$ is closed. Since $C$ is homotopy equivalent to a circle, it suffices to check that $\int \beta = 0$ over any generator of $\pi_1(C)$. In particular, we may use the periodic orbit. However, since $\beta(X) \equiv 0$, it vanishes on the periodic orbit. Hence, $\beta$ is exact, and we conclude the Lemma.
 \end{proof}

\begin{remark}
    In the case of a hyperbolic fixed points, local normal forms are conjugated if their Taylor expansions coincide \cite[Theorem 6.6.5]{KH}. However, when both maps preserve the standard Lebesgue measure, it is unlikely that such a conjugacy preserves it. Since the contact forms ``see'' the volume, we expect that coincidence of the Taylor expansions to be insufficient to conclude the contact forms are related by pullback without an assumption that the volumes coincide. 
\end{remark}

We now proceed to prove Theorem \ref{thm E}. This is done over the course of three lemmas. In Lemma \ref{pullback-to-data}, we show that if two contact forms are related by pullback, all of the dynamical data is preserved. In particular, assumption (1) of Theorem \ref{thm E} implies both assumptions (2) and (3). In Lemma \ref{return-to-pullback}, we show that if (2) of Theorem \ref{thm E} is satisfied, then (1) is. Finally, in Lemma \ref{time-to-pullback}, we show that assumption (3) implies (1).

\begin{lemma}
    \label{pullback-to-data}
  For $i=1,2$, let $\Phi_i=(\phi_i^t)$ be the $C^\infty$ Reeb flow associated to a contact form $\alpha_i$ in a neighborhood $U_i$ of its hyperbolic periodic orbit $\mc O_i$.  
Assume that there exists $H\colon U_1 \to U_2$ such that $H^*\alpha_2 = \alpha_1$. Then there exist sections $\Sigma_1$ and $\Sigma_2$ of the associated Reeb flows transverse to $\mc O_1$ and $\mc O_2$, respectively, such that:

    \begin{itemize}
        \item $H(\Sigma_1) = \Sigma_2$;
        \item $\mc O_1$ and $\mc O_2$ have the same prime period and Lyapunov exponents;
        \item The Poincar\`{e} maps $F_1$ and $F_2$ are conjugated;
        \item There exist charts $h_i\colon D^2_\ve \to \Sigma_i$, $i= 1,2$ such that $h_i^*\alpha_i = \frac{1}{2}(x \, dy - y \, dx)$;
        \item If $\hat{r}_i\colon \Sigma_i \to \R$ denotes the return time of the Reeb flows, and $r_i = \hat{r}_i \of h_i$, then $r_1 = r_2$.
    \end{itemize}
\end{lemma}

\begin{proof}
    This all follows from the fact that $H$ intertwines the Reeb flows of $\alpha_1$ and $\alpha_2$, which we denote by $\varphi_1^t$ and $\varphi_2^t$, respectively. Indeed, note that if $X_1$ is the Reeb field of $\alpha_1$, then $\alpha_2(H_*(X_1)) = H^*\alpha_2(X_1) = \alpha_1(X_1) \equiv 1$, and similarly $\iota_{H_*X_1}d\alpha_2 \equiv 0$. Then $H_*X_1 = X_2$ and $H$ conjugates $\varphi_1$ and $\varphi_2$. As a result, the prime period and Lyapunov exponents coincide. 

    Let $H_0\colon C \to U_1$ be such that $H_0^*\alpha_1$ is a resonance contact form, $H_0^*\alpha_1 = \theta_1(xy) \, dt + \frac{1}{2}(x \, dy - y \, dx)$. Let $\Sigma_1 = H_0(\set{0}\times D^2_{\ve})$ and $\Sigma_2 = H(\Sigma_1)$. By construction, since $H$ intertwines the Reeb flows, the Poincar\`{e} maps of $\Sigma_1$ and $\Sigma_2$, $F_1$ and $F_2$ are conjugated. Similarly, by construction of $H_0$, if $h_1 = H_0|_{D^2_\ve}$ and $h_2 = H \of h_1$, it follows that $h_1^*\alpha_1 = \frac{1}{2}(x \, dy  - y \, dx)$ and $h_2^*\alpha_2 = (H \of h_1)^*\alpha_2 = h_1^*H^*\alpha_2 = h_1^*\alpha_1 = \frac{1}{2}(x \,dy - y \, dx)$. Finally, note that if $\varphi_1^{r_1(\vec z)}(h_1(\vec z)) \in \Sigma_1$ for all $\vec z \in D^2_\ve$, then noting that $h_2 = H \of h_1$,

    \[ \varphi_2^{r_2(\vec z)}(h_2(\vec z)) = \varphi_2^{\hat{r}_2(h_2(\vec z))}(h_2(\vec z)) = \varphi_2^{\hat{r}_2(H(h_1(\vec z))} (H(h_1(\vec z))) = H(\varphi_1^{\hat{r}_2(H(h_1(\vec z)))}(h_1(\vec z)))\]

    Note that this belongs to $\Sigma_2$ if and only if $\varphi^{\hat{r}_2(H(h_1(\vec z)))}_1(h_1(\vec z)) \in \Sigma_1$. The smallest such value is by definition $r_1(\vec z)$. Hence, $r_2(\vec z) = r_1(\vec z)$ for all $\vec z\in D^2_\ve$.
    \end{proof}

\begin{lemma}
    \label{return-to-pullback}
    For $i=1,2$, let $\Phi_i=(\phi_i^t)$ be the $C^\infty$ Reeb flow associated to a contact form $\alpha_i$ in a neighborhood $U_i$ of its hyperbolic periodic orbit $\mc O_i$. Let $\Sigma_i$ be a section of $\Phi_i$ transverse to $\mc O_i$, $i=1,2$. If the prime periods of $\mc O_1$ and $\mc O_2$ coincide, and the Poincar\`{e} return maps of $\Sigma_1$ and $\Sigma_2$ are conjugated by a $C^\infty$ diffeomorphism $H_\Sigma\colon \Sigma_1\rightarrow \Sigma_2$ such that ${H_\Sigma}^*d\alpha_2 = d\alpha_1$, then there exists $H\colon U_1 \to U_2$ defined in a neighborhood of $\mc O_1$ such that $H^*\alpha_2 = \alpha_1$. 
\end{lemma}

\begin{proof}
  By Theorem~\ref{thm:contact-normal}, we can choose a coordinate system $H_1\colon C \to U_1$ such that $H_1^*\alpha_1 = \theta_1(xy)\, dt + \frac{1}{2}(x \, dy - y \, dx)$, where $C=\mathbb R/\mathbb Z\times D^2_\ve$ for some $\ve>0$. By Proposition~\ref{prop:cocycle-data}(4), the first return map $F_1$ to section $H_1(\{0\}\times D^2_\ve)$ is time-1 map of the vector field $V_{\theta_1} = {\theta_1}'(xy)(x \, \frac{\partial}{\partial x} - y \, \frac{\partial}{\partial y})$. We note that the Poincar\`{e} return maps of any two sections $\Sigma_1$ and $\tilde \Sigma_1$ transverse to $\mc O_1$ at the same point are always locally conjugated by a map $\pi_{\Sigma_1,\tilde\Sigma_1}\colon \Sigma_1\rightarrow \tilde\Sigma_1$ defined through moving along the Reeb flow from $\Sigma_1$ to $\tilde \Sigma_1$. Moreover, the Reeb flow direction is in the kernel of $d\alpha_1$ so $\pi_{\Sigma_1,\tilde\Sigma_1}^*d\alpha_1|_{\tilde\Sigma_1}=d\alpha_1|_{\Sigma_1}$. Thus, we may assume without loss of generality that $\Sigma_1 = H_1(\set{0} \times D^2_\ve)$. Denote $h_1 = H_1|_{\set{0}\times D^2_\ve}$.

    Let $H_\Sigma\colon \Sigma_1 \to \Sigma_2$ be a map as in the lemma for some section $\Sigma_2$. Then $h_2 := H_\Sigma \of h_1$ puts the first return map of $\Phi_2$ to section $\Sigma_2$ into Birkhoff normal form with the same return map $F_1$. Furthermore, ${h_2}^*d\alpha_2 = h_1^*H_\Sigma^*d\alpha_2 = h_1^*d\alpha_1 = dx \wedge dy$. It follows from Lemma \ref{lem:canonnical-section} that one may modify the section $\Sigma_2$ by moving along the flow to produce a parameterized section which has the same expression of first return map written using that parametrization, and such that the pullback of $\alpha_2$ under this parametrization is  $\frac{1}{2}(x \,dy - y\, dx)$. Thus, we can assume that $h_2^*\alpha_2=\frac{1}{2}(x \,dy - y\, dx)$. Since the first return map to $\Sigma_2$ is now given by $F_1$ in the coordinates provided by $h_2$, following Sections \ref{sec:contact-cocycle} and \ref{sec:contact-build}, one builds a coordinates $H_2\colon C \to U_2$ around $\mc O_2$ such that $H_2^*\alpha_2 = \theta_1(xy)\, dt + \frac{1}{2}(x\, dy - y\, dx)$. Then $H = H_2 \of H_1^{-1}$ satisfies $H^*\alpha_2 = (H_1^{-1})^*H_2^*\alpha_2 = (H_1^{-1})^*(\theta_1(xy)\, dt + \frac{1}{2}(x\, dy - y \,dx))= \alpha_1$.
\end{proof}

\begin{lemma}
    \label{time-to-pullback}
    For $i=1,2$, let $\Phi_i=(\phi_i^t)$ be the $C^\infty$ Reeb flow associated to a contact form $\alpha_i$ in a neighborhood $U_i$ of its hyperbolic periodic orbit $\mc O_i$. Let $\Sigma_i$ be a section of $\Phi_i$ transverse to $\mc O_i$, $i=1,2$.
    Assume that
    
    \begin{itemize}
        \item $\mc O_1$ and $\mc O_2$ have the same Lyapunov exponents,
        \item there exist parameterizations $h_i\colon D^2_\ve \to \Sigma_i$ such that $h_i^*\alpha_i = \frac{1}{2}(x \, dy - y \, dx)$,
        \item the first return maps of the Reeb flows are in resonance form in the coordinates provided by $h_i$.\footnote{Such parameterizations always exist by Lemma \ref{lemma: nice section}}
    \end{itemize} Then if $\hat{r}_i$ is the return time of the section $\Sigma_i$, $r_i = \hat r_i \of h_i$, and $r_1 = r_2$, there exists $H\colon U_1 \to U_2$ defined in a neighborhood of $\mc O_1$ such that $H^*\alpha_2 = \alpha_1$.
\end{lemma}

\begin{proof}
   Notice that our assumptions are exactly those which appear in Lemma~\ref{lemma: nice section}. 
   It follows that if $F_i(x,y) = (e^{\eta_i(xy)}x,e^{-\eta_i(xy)}y)$ are the first return maps of the sections $\Sigma_i$ in the coordinates provided by $h_i$, $-xy \,d\eta_1 = dr_1 = dr_2 = -xy\, d\eta_2$. Hence, $d\eta_1 = d\eta_2$ (we may divide by $-xy$ since the Foulon-Hasslelbatt cocycle vanishes). It follows that $\eta_1 = \eta_2 + c$. Since the positive Lyapunov exponent of $\mc O_i$ is given by $\eta_i(0)/r_i(0)$ for $i = 1,2$, and we know that $r_1(0) = r_2(0)$ by assumption, we conclude that $\eta_1(0) = \eta_2(0)$. It follows that $\eta_1 = \eta_2$.

    As in the previous Lemma, we conclude following Section \ref{sec:contact-build} that $\alpha_1$ and $\alpha_2$ have identical normal forms coordinates, and hence are related by pullback.
\end{proof}

\subsection{Linearizable contact forms}

In this subsection, we discuss contact forms that can be written in a special resonance form that we call {\it linear}. In particular, we present the proof of Theorem \ref{thm:linear-rigidity}. 


Consider $C = \R / \Z \times D^2_\ve$ with standard coordinates $(t,x,y)$. Let $\lambda>1$ and $r_0>0$ be some constants.

\begin{enumerate}
    \item A local diffeomorphism is linear if $F(x,y) = (\lambda x,\lambda^{-1}y)$ (i.e., if $\omega \equiv 1$)
    \item A conservative vector field on $C$ is linear if $X = r_0^{-1}\left(\frac{\del}{\del t} + \lambda x \frac{\del}{\del x} - \lambda y \frac{\del}{\del y}\right)$ (i.e., if $f$ and $g$ are constant functions)
    \item A hyperbolic contact form on $C$ is linear if $\alpha = (r_0 + \lambda xy)\, dt + \frac{1}{2}(x \, dy - y \, dx)$
\end{enumerate}

If there exist coordinates in which an object is linear, we call it {\it linearizable}. 

\begin{example}
    If $\Phi$ is a constant-time suspension of an Anosov diffeomorphism $f$, the roof function is linearizable for any section, and the vector field around periodic orbits is linearizable if and only if corresponding periodic orbit of $f$ is linearizable.
\end{example}

\begin{example} \label{ex:homogeneous}
    Consider a quotient of $SL(2,\R)$ by a discrete group $\Gamma$, and assume that $g\Gamma \in SL(2,\R)/\Gamma$ is a periodic orbit of the geodesic flow $g_t = \begin{pmatrix}
        e^{t/2} & 0 \\ 0 & e^{-t/2}
    \end{pmatrix}$ (i.e., the left translation action by the subgroup $A = \set{ g_t : t \in \R}$). Without loss of generality, since the flow is homogeneous, we may assume that $g = e$, in which case we can conclude that $g_T(e\Gamma) = e\Gamma$, and $\gamma =\begin{pmatrix}
        e^{T/2} & 0 \\ 0 & e^{-T/2}
    \end{pmatrix} \in \Gamma$. Consider the local section determined by the following matrices: \[p(x,y) = \begin{pmatrix}
        \sqrt{1+xy} & x \\
        y & \sqrt{1+xy}
    \end{pmatrix} \qquad\Sigma = \set{ p(x,y) : \abs{x},\abs{y} < 1}.\] Then the return time of the section is constant and equal to $T$, since a direct computation gives \[ g_T p(x,y) = p(e^Tx,e^{-T}y)\gamma \sim p(e^Tx,e^{-T}y).\]
    Further computation show that the coordinates $H(x,y,t) = \begin{pmatrix}
        e^{tT/2}\sqrt{1+xy} & e^{-tT/2}x \\ e^{tT/2}y & e^{-tT/2}\sqrt{1+xy}
    \end{pmatrix}$, makes the right action of $\gamma$ is indeed translation in the $t$ coordinate by $1$. Furthermore, using that $\ker\alpha = E^s \oplus E^u$ and $\alpha(X) \equiv 1$, one obtains that the corresponding contact form is the linear one
    \[ \alpha_0 = T(1 + xy)\, dt + \frac{1}{2}(x \, dy - y \, dx).\]
\end{example}

\begin{lemma}
\label{lem:std-area-form}
    If $\omega$ is a $C^\infty$ nonvanishing 2-form on $D_\ve^2$ invariant under $F(x,y) = (\lambda x,\lambda^{-1}y)$, there exists $0 < \delta < \ve$ and $C^\infty$ diffeomorphism $H\colon D_\delta^2 \to D_\ve^2$ such that $H$ commutes with $F$ and $H^*\omega = dx \wedge dy$.
\end{lemma}

\begin{proof}
    Consider a 2-form $\omega = g \, dx \wedge dy$ for some strictly positive $g \in C^\infty(D_\ve^2)$. Note that $\omega$ is $F$-invariant if and only if $g$ is $F$-invariant. 
    
    We apply the Moser trick. We will build a flow $\psi^t$ generated by a bounded time-dependent vector field $Y_t$ such that if $\omega_t = [(1-t)g + t] \,dx \wedge dy$, then $\omega = (\psi^t)^*\omega_t$. This occurs if and only if $\mc L_{Y_t}\omega_t +(1-g)\, dx \wedge dy = 0$, or $d\left(\iota_{Y_t}[(1-t)g+ t]\, dx \wedge dy\right) = (g-1)\, dx \wedge dy$. We seek an $F$-invariant solution of the form $Y_t = f_t\, \del /\del x$ in which case we seek to solve
    \begin{eqnarray*} d\left((1-t)g+t)f_t \,dy\right) &  = & (g-1)\, dx \wedge dy \\
    \frac{\del\left[((1-t)g+t)f_t\right]}{\del x} \, dx \wedge dy & = & (g-1)\,dx \wedge dy\end{eqnarray*}
    Let $f_t(x,y)$ be defined by
    \[f_t(x,y) = \dfrac{1}{(1-t)g+t}\left(\int_0^xg(s,y)\,ds-x\right)\]
    Then $\del([(1-t)g+t)f_t]/\del x = g(x,y) - 1$, so $f_t$ solves the desired equation. $f_t$ is well-defined and $C^\infty$ since $g$ is always positive and integration only improves regularity. Furthermore, if $Y_t = f_t \, \del / \del x$, $F_*Y_t = Y_t$ if and only if $f_t \of F = \lambda f_t$. Notice that $g$ is $F$-invariant so the fraction term of $f_t$ is $F$-invariant. Furthermore, after change of coordinates and again using $F$-invariance of $g$, the integral term satisfies
    \[ \int_0^{\lambda x} g(s,\lambda^{-1}y)\,ds - \lambda x = \int_0^x g(\lambda t,\lambda^{-1}y)\cdot \lambda \, dt - \lambda x = \lambda\left(\int_0^xg(t,y)\,dt-x\right)\]
    It follows that $f_t \of F = \lambda f_t$, and hence we have shown the lemma.
\end{proof}
\begin{remark}\label{rem:nonlin-centralizer}We note here that Lemma~\ref{lem:std-area-form} does not work for nonlinear $F$, which is shown in the following example. In particular, we do not believe that the assumptions of Theorem~\ref{thm E} can be weakened to those of Theorem \ref{thm:linear-rigidity} in general.
\end{remark}
\begin{example}\label{ex:nonlinear-centralizer}
    Let $F(x,y)=(2x+x^2y, \frac{y}{2+xy})$ on $\mathbb R^2$. Since $\det(DF)\equiv 1$, $DF$ preserves the following volume forms $c\, dx\wedge dy$ for any conestant $c>0$. In particular, it preserves, the volume form $2 \, dx\wedge dy$. One should not think of this as the same volume form, as the contact form determines the area form on the transversal completely, not just up to scalar multiple.
    We claim that there is no $C^\infty$ diffeomorphism such that $H\circ F=F\circ H$ and $H^{\ast}(2 \,dx\wedge dy)=dx\wedge dy$. 
    
    First, we note that  $\det(DH(x,y))=\frac{1}{2}$ as $H^{\ast}(2\, dx\wedge dy)=dx\wedge dy$.

    Moreover, since $F^n(x,0)=(2^nx,0)$ and $F^n(0,y)=(0,2^{-n}y)$ for all $n\in\mathbb N$, we obtain that $H(x,0)=(f_1(x),0)$ and $H(0,y)=H(0,f_2(y))$ for all $x,y\in\mathbb R$ for some smooth functions $f_1,f_2$ on $\mathbb R$ such that $f_1(0)=f_2(0)=0$. Then, we can write $H(x,y)=(H_1(x,y),H_2(x,y))$ in the form $$H(x,y)=(xh_1(x,y), yh_2(x,y)),$$ 
    where $h_1(x,y) = H_1(x,y)/x = \int_0^1\frac{\partial H_1}{\partial x}(sx,y) \, ds$ and $h_2(x,y) =  H_2(x,y)/y =\int_0^1\frac{\partial H_2}{\partial y}(x,sy) \, ds$. In particular, $h_1,h_2$ are $C^\infty$ functions since $H_1(0,y) = 0 = H_2(x,0)$ and $H_1$ and $H_2$ are both $C^\infty$.
From 
$H\circ F=F\circ H$, we obtain the following relation:
$$(2+xy)h_1\left(2x+x^2y, \frac{y}{2+xy}\right)=2h_1(x,y)+xh_1^2(x,y)yh_2(x,y).$$
Differentiating with respect to $x$ gives us
\begin{eqnarray*}
&&yh_1\left(2x+x^2y, \frac{y}{2+xy}\right)+(2+xy)\Bigg[\frac{\partial h_1}{\partial x}\left(2x+x^2y, \frac{y}{2+xy}\right)(2+2xy) \\
&&\hspace{3.5cm} -\frac{\partial h_1}{\partial y}\left(2x+x^2y, \frac{y}{2+xy}\right)\frac{y^2}{(2+xy)^2}\Bigg]\\
&&=2\frac{\partial h_1}{\partial x}(x,y)+h^2_1(x,y)yh_2(x,y)+x\frac{\partial\left(h_1^2(x,y)yh_2(x,y)\right)}{\partial x}(x,y).
\end{eqnarray*}
Then differentiating above equation with respect to $y$ and taking value at $(0,0)$ gives us 
$$h_1(0,0)=h_1(0,0)^2h_2(0,0),$$
which gives us either $h_1(0,0)=0$ or $h_1(0,0)h_2(0,0)=1$ so 
    $\det(DH(0,0))=0$ or $1$, respectively, which is a contradiction to $\det(DH(x,y))=\frac{1}{2}$. 
\end{example}

In the other direction, the following example illustrates that when the structures are $C^\infty$ tangent to linear models, we can solve the corresponding cohomology equations (i.e., this is a case of Theorem \ref{thm:linear-rigidity}).

\begin{example}
    Let $\theta(z) = 1+z+e^{-1/z^2}$ and $$\alpha = \theta(xy)\, dt + \frac{1}{2}(x\,dy - y\, dx)= (1+xy+e^{-1/x^2y^2})\, dt +\frac{1}{2}(x \, dy - y \, dx).$$ Then the return time is given by $r(x,y) = \theta(xy) -xy\theta'(xy) = 1+e^{-1/x^2y^2}(1-2/x^3y^3) =\bar{r}(xy)$, where $\bar{r}(z) = e^{-1/z^2}(1-2/z^3)$. Then $r$ is cohomologous to a constant, with transfer function $u(x,y) = -\dfrac{\bar{r}(xy)-1}{\theta'(xy)}\log \abs{x}$, since $u \of F = -\dfrac{\bar{r}(xy)-1}{\theta'(xy)}[\log (\abs{xe^{\theta'(xy)}})]$, so $u \of F - u = \bar{r}(xy)-1$.
\end{example}

\begin{proof}[Proof of Theorem \ref{thm:linear-rigidity}]
    We show that (a) implies (d) implies (e) implies (c) implies (b) implies (a). That (a) implies (d) is immediate from Proposition \ref{prop:cocycle-data} and the definition of a linearizable contact form. (d) implies (e) is immediate.

    We now show that (e) implies (c). Indeed, if some section $\Sigma_0$ has its roof function cohomologous to a function which is $C^\infty$ tangent to a constant, then $\Sigma_\tau$ has its roof function cohomologous   to a function which is $C^\infty$ tangent to a constant with transfer function $\tau$ (where $\Sigma_\tau$ is as in \eqref{eq: sigma_tau}). In particular, for {\it any} smooth transverse section, the roof function is $C^\infty$ cohomologous to a function $C^\infty$ tangent to a constant.
    
    Putting the contact form in the normal form of Theorem \ref{thm:contact-normal}, we see that the corresponding function $\bar{r}(z) = \theta(z) - z\theta'(z)$ is $C^\infty$ tangent to a constant by Lemma \ref{lem:cocycle-coefficeints}. If $\theta$ has Taylor jet $\theta(z) \approx \sum_{k=0}^\infty a_kz^k$, then $\bar{r}(z) \approx a_0 + \sum_{k=1}^\infty (1-k)a_kz^k$. It follows that if $\bar{r}$ is $C^\infty$ tangent to a constant, then $a_k = 0$ for $k \ge 2$. Therefore, $\theta'(z)$ is $C^\infty$ tangent to a constant, and the first return map is $(x,y) \mapsto (e^{\theta'(xy)}x,e^{-\theta'(xy)}y)$ is $C^\infty$ tangent to the linear map. Thus we conclude (c).
    
    That (c) implies (b) is exactly \cite[Theorem 6.6.5]{KH}.

    Finally, we show that (b) implies (a). Assume that the first return map for some parameterized section $H_0\colon D^2_\ve \to \Sigma$ is linear. Then $d\alpha|_\Sigma$ is invariant under $F$, so by Lemma~\ref{lem:std-area-form}, one may pick a local diffeomorphism  $H_1\colon D^2_\delta \to D^2_\ve$ which commutes with $F$ such that $H_1^*H_0^*d\alpha|_\Sigma = dx \wedge dy$. It follows that if $H = H_0 \of H_1$, then the first return map induced by $H$ is still $F$, and $H^*d\alpha|_\Sigma = dx \wedge dy$. It follows from the equivalence of (1) and (2) in Theorem \ref{thm E} that some pullback of $\alpha$ is a linear form (simply pick the linear form with the correct return map and period as a model). Therefore, $\alpha$ is linearizable, and (b) implies (a). 
\end{proof}

\bibliographystyle{alpha}
\bibliography{refs}
\end{document}